\documentclass[english,11pt]{article}
\pdfoutput=1
\usepackage[a4paper,bottom=2cm,top=2cm,left=2.54cm,right=2.54cm]{geometry}
\usepackage[english]{babel}
\usepackage[utf8]{inputenc}
\usepackage{url}
\usepackage[usenames]{color}
\usepackage{amsmath}
\usepackage{amssymb}
\usepackage{amsthm}
\usepackage{color}
\usepackage{graphicx}
\usepackage{subcaption}
\usepackage{yfonts}

\usepackage[backend=bibtex,maxnames=4,url=false,style=numeric,isbn=false, giveninits=true,sorting=nyt]{biblatex}
\bibliography{abbrevs,BL,constr-dti,IP}

\newcommand{\subdiff}{\partial}
\newcommand{\defeq}{:=}

\DeclareMathOperator*{\conv}{conv}
\newcommand{\norm}[1]{\|#1\|}

\DeclareMathOperator{\TV}{TV}
\DeclareMathOperator{\TGV}{TGV}

\DeclareMathOperator{\PSNR}{PSNR}
\DeclareMathOperator{\SSIM}{SSIM}
\newcommand{\weakto}{\mathrel{\rightharpoonup}}

\newcommand{\R}{\mathcal R}
\newcommand{\A}{\textgoth A}
\newcommand{\F}{\mathfrak F}
\newcommand{\U}{\mathfrak U}
\renewcommand{\leq}{\leqslant}
\renewcommand{\geq}{\geqslant}

\newtheorem{theorem}{Theorem}

\newtheorem{lemma}[theorem]{Lemma}

\theoremstyle{definition}

\newtheorem*{assumption*}{Assumption}
\newtheorem{remark}[theorem]{Remark}
\newtheorem*{remark*}{Remark}
\newtheorem*{definition*}{Definition}

\numberwithin{equation}{section}
\numberwithin{theorem}{section}

\title{Image reconstruction with imperfect forward models and applications in deblurring}
\author{Yury Korolev$^*$ \and Jan Lellmann\thanks{\{korolev,lellmann\}@mic.uni-luebeck.de}}
\date{Institute of Mathematics and Image Computing, University of L\"ubeck \\
Maria-Goeppert-Str.~3, 23562 L\"ubeck, Germany}

\begin{document}

\maketitle

\begin{abstract}
We present and analyse an approach to image reconstruction problems with imperfect forward models based on partially ordered spaces -- Banach lattices. In this approach, errors in the data and in the forward models are described using order intervals. The method can be characterised as the lattice analogue of the residual method, where the feasible set is defined by linear inequality constraints. The study of this feasible set is the main contribution of this paper. Convexity of this feasible set is examined in several settings and modifications for introducing additional information about the forward operator are considered. Numerical examples demonstrate the performance of the method in deblurring  with errors in the blurring kernel.
\end{abstract}


\section{Introduction}

The goal of image reconstruction is obtaining an image of the object of interest from indirectly measured, and typically noisy, data. Mathematically, image reconstruction problems are commonly formulated as inverse problems that can be written in the form of operator equations
\begin{equation}\label{Au=f}
Au = f,
\end{equation}
\noindent where $u \in \mathcal U$ is the unknown, $f \in \mathcal F$ is the measurements and $A\colon \mathcal U \to \mathcal F$ is a forward operator that models the data acquisition. In this paper, we consider linear forward operators and assume that equation~\eqref{Au=f} with exact data and operator has a unique solution that we denote by $\bar u$.

In practice, not only the right-hand side $f$ is noisy, but also the operator $A$ is often not exact as it contains errors that come from imperfect calibration measurements.

Uncertainty in the operator and in the data may be characterised by the inclusions $A \in \A$ and $f \in \F$ for some sets $\A \subset L(\mathcal U,\mathcal F)$ and $\F \subset \mathcal F$. These sets may be referred to as \emph{uncertainty sets} -- a concept widely used in robust optimisation~\cite{Bertsimas_rob_opt_2011}. Given $\A$ and $\F$, we would like to find a subset $\U \in U$, called the feasible set, that contains the exact solution $\bar u$. Depending on the particular form of the uncertainty sets $\A$ and $\F$ and on available additional a priori information about $\bar u$, the inclusion $\bar u \in \U$ can be proven for different feasible sets. Two main considerations that affect the choice of a particular feasible set are its size (smaller feasible sets are preferred) and the availability of efficient numerical algorithms for optimisation problems involving the feasible set (therefore, convex feasible sets are preferred). 

For ill-posed problems, in general, available feasible sets are too large and contain elements arbitrary far from $\bar u$. An exception is the case when a compact set that contains the exact solution $\bar u$ is known a priori~\cite{TGSYag}. In this case, a feasible set of finite diameter can be obtained. In the general case, an appropriate regularisation functional $\R$ needs to be minimised on the feasible set to find a stable approximation to $\bar u$.

This is the idea behind the residual method~\cite{IvanovVasinTanana, GrasmairHalmeierScherzer2011}. Operating in normed spaces, one can define the uncertainty sets as follows:
\begin{equation}\label{uncertainty_sets_norm}
\F := \{f \in \mathcal F \colon \norm{f-f_\delta} \leq \delta\}, \quad \A := \{A \in L(\mathcal U,\mathcal F) \colon \norm{A-A_h} \leq h\}
\end{equation}
\noindent for an approximate right-hand side $f_\delta$, approximate forward operator $A_h$ and approximation errors $\delta$ and $h$~\cite{IvanovVasinTanana}. Using the information in~\eqref{uncertainty_sets_norm}, one can define a feasible set as follows~\cite{IvanovVasinTanana}:
\begin{equation}\label{U_h_delta}
U_{h,\delta} = \{u \in \mathcal U \colon \norm{A_h u - f_\delta} \leq \delta + h\norm{u}\}.
\end{equation}
\noindent This set contains all elements of $\mathcal U$ that are consistent with~\eqref{Au=f} within the tolerances given by~\eqref{uncertainty_sets_norm}. The inclusion $\bar u \in U_{h,\delta}$ can be easily verified. Unless $h=0$ (i.e. the forward operator is exact), the set $U_{h,\delta}$ is non-convex and the residual method results in a non-convex optimisation problem even for convex regularisation functionals.

An alternative approach to modelling uncertainty in $A$ and $f$ using partially ordered spaces was proposed in~\cite{Kor_IP:2014, Kor_Yag_IP:2013}. Assume that $\mathcal U$ and $\mathcal F$ are Banach lattices, i.e. Banach spaces with partial order ``$\leq$'', and that $A$ is a regular operator~\cite{Schaefer}. Then, uncertainties in $A$ and $f$ can be characterised using intervals in appropriate partial orders, i.e.
\begin{equation}\label{bounds}
\F = \{f \in \mathcal F \colon f^l \leq f \leq f^u\}, \quad \A = \{A \in L^r(\mathcal U,\mathcal F) \colon A^l \leq A \leq A^u \},
\end{equation}
\noindent where $L^r(\mathcal U, \mathcal F) \subset L(\mathcal U, \mathcal F)$ is the space of regular operators $\mathcal U \to \mathcal F$. Assuming positivity of the exact solution $\bar u$ and using the inequalities in~\eqref{bounds}, we can show that the exact solution $\bar u$ is contained in the following feasible set~\cite{Kor_IP:2014}:
\begin{equation}\label{U}
U:=\{u \in \mathcal U \colon u \geq 0, \,\, A^l u \leq f^u, \,\, A^u u \geq f^l\}.
\end{equation}
In contrast to~\eqref{U_h_delta}, the set in \eqref{U} is convex and minimising a convex regularisation functional $\R$ on this set results in a convex optimisation problem:
\begin{equation}\label{optimisation_problem}
\min_{u\in \mathcal U} \R(u) \quad \text{s.t. } u \geq 0, \,\, A^l u \leq f^u, \,\, A^u u \geq f^l. 
\end{equation}
\noindent Using the relationship between partial orders and norms in Banach lattices, one can prove the inclusion of the partial-order-based feasible set $U$~\eqref{U} in the norm-based feasible set $U_{h,\delta}$~\eqref{U_h_delta} for appropriately chosen $A_h, u_\delta, h, \delta$.  We briefly review the partial-order-based approach in Section~\ref{section-IP_in_BL}.

While convergence of the minimisers of~\eqref{optimisation_problem} to the exact solution can be guaranteed~\cite{Kor_IP:2014}, it is not clear, whether the solution to \eqref{optimisation_problem} actually corresponds to a particular pair $(A,f)$ within the bounds~\eqref{bounds}. It seems more natural to look for approximate solutions in the following set:
\begin{equation}\label{U*}
U^* = \{u \in \mathcal U \colon u \geq 0, \,\, \exists A, \,\, A^l  \leq A \leq A^u, \,\, \exists f, \,\, f^l \leq f \leq f^u, \,\, Au=f\},
\end{equation}
i.e. to assume that, while the exact operator and noise-free measurements are unknown, there has to exist at least one pair within the uncertainty bounds~\eqref{bounds} that exactly explains the solution.

It is not clear a priori, whether the set $U^*$~\eqref{U*} is convex. In this paper we show that the sets $U$~\eqref{U} and $U^*$~\eqref{U*}, in fact, coincide, which implies convexity of $U^*$ (see Section~\ref{section-U=U*}) and shows that the convex problem~\eqref{optimisation_problem} actually implements the natural formulation~\eqref{U*}.

It is tempting to add an additional constraint on the operator in~\eqref{U*}, reflecting additional \emph{a priori} information about $A$. For instance, if $A$ is a blurring operator, then (after finite-dimensional approximation) the rows of the matrix $A$ should sum up to one, i.e. we should have that $Ae=e$ for two vectors of ones of appropriate lengths. More generally, one can define an additional linear constraint $Av=g$ for a fixed pair $(v,g)$ to obtain

\begin{equation}\label{U**}
\begin{aligned}
U^{**} = \{u \in \mathcal U \colon u \geq 0, \,\, \exists A, \,\, A^l  \leq A \leq A^u, \,\, Av=g, \\ \exists f, \,\, f^l \leq f \leq f^u, \,\, Au=f\} \subset U^*.
\end{aligned}
\end{equation}

Unfortunately, even such a simple constraint breaks the convexity of the feasible set. We demonstrate this in Section~\ref{section-U**} by explicitly describing the set~\eqref{U**} in finite dimensions in the special case when $f^l = f^u$. We also argue that the additional constraint $Av=g$ can still be useful to tighten the bounds $A^l, A^u$ if they weren't carefully chosen initially.

Since the set $\{A \colon Av=g\}$ is convex (in $A$), the analysis of Section~\ref{section-U**} shows that convexity of an additional constraint set $\mathcal A \subset \{A \colon A^l \leq A \leq A^u\}$ does not guarantee convexity of the corresponding feasible set in $u$. Because of this negative result, we confine ourselves to the set~\eqref{U} in our numerical experiments.

In Section~\ref{section-deblurring} we consider an application in image deblurring with uncertainty in the blurring kernel. In many applications, such as astronomy or fluorescence microscopy, the blurring kernel (often referred to as the point spread function) is obtained experimentally by recording light from reference stars~\cite{Telescope_PSF} or imaging subresolution fluorescent particles~\cite{Shaw_PSF_91}. Such blurring kernels inevitably contain errors that can significantly impact the reconstruction. Blind deconvolution~\cite{Chan_Wong_blind_deconvolution_98, Perrone_Favaro_blind_deconvolution_16} aims at reconstructing both the blurring kernel and the image simultaneously, but suffers from severe ill-posedness and non-convexity. The approach we propose takes into account the errors in the available blurring kernel (without attempting to obtain a better estimate of it) while staying within the convex setting.

\section{Brief overview of the partial-order-based approach}\label{section-IP_in_BL}
$L_p$ spaces, endowed with a partial order relation
\begin{equation*}
f \leq g \text{ iff } f(\cdot) \leq g(\cdot) \text{ a.e.},
\end{equation*}
\noindent become Banach lattices, i.e. partially ordered Banach spaces with well-defined suprema and infima of each pair of elements and a monotone norm~\cite{Schaefer}. If $\mathcal E$ and $\mathcal F$ are two Banach lattices, then partial orders in $\mathcal E$ and $\mathcal F$ induce a partial order in  a subspace of the space of linear operators acting from $\mathcal E$ to $\mathcal F$, namely in the space of regular operators. A linear operator $A \colon \mathcal E \to \mathcal F$ is called regular, if it can be represented as a difference of two positive operators. Positivity of an operator is defined as $A \geq 0$ iff $\forall x \in \mathcal E \,\, x \geq 0 \implies Ax \geq 0$. Partial order in the space of regular operators is introduced as follows: $A \geq B$ iff $A-B$ is a positive operator.  Every regular operator acting between two Banach lattices is continuous~\cite{Schaefer}.

 The framework of partially ordered functional spaces allows quantifying uncertainty in the data $f$ and forward operator $A$ of the inverse problem~\eqref{Au=f} by means of order intervals~\eqref{bounds}. Approximate solutions to~\eqref{Au=f} are the minimisers of~\eqref{optimisation_problem}. Convergence of these minimisers to the exact solution of~\eqref{Au=f} is studied as the uncertainty in the data $f$ and forward operator $A$ diminishes. This is formalised using monotone convergence sequences of lower and upper bounds~\eqref{bounds}:
 \begin{equation}\label{bounds_sequences}
 \begin{aligned}
& f^l_n, f^u_n \colon           \quad &f^l_n \leq f \leq f^u_n,         \quad   &f^l_{n+1} \geq f^l_n, \,\, f^u_{n+1} \leq f^u_n,          \quad   &\norm{f^u_n-f^l_n} \to 0, \\
&A^l_n, A^u_n \colon    \quad &A^l_n \leq A \leq A^u_n,         \quad   &A^l_{n+1} \geq A^l_n, \,\, A^u_{n+1} \leq A^u_n,  \quad   &\norm{A^u_n-A^l_n} \to 0.
\end{aligned}
\end{equation}
\noindent If a sequence of lower (upper) bounds is not monotone, it can always be made monotone by consequently taking the supremum (infimum) of each element in the sequence with the preceding one.

Convergence of the corresponding sequence of minimisers $u_n$ of~\eqref{optimisation_problem} to $\bar u$ is guaranteed by the following theorem~\cite{Kor_IP:2014}:
\begin{theorem}\label{thm-convergence}
Let $\mathcal U$ and $\mathcal F$ be Banach lattices and $\mathcal F$ order complete\footnote{A Banach lattice $\mathcal U$ is called order complete if every majorised set in $\mathcal U$ has a supremum.}. Suppose that the regulariser $\mathcal R$ satisfies the following assumptions:
\begin{itemize}
\item $\mathcal R$ is bounded from below on $\mathcal U$; 
\item $\mathcal R$ is lower semi-continuous; 
\item non-empty sub-level sets   $lev_C(\mathcal R) = \{u \colon \mathcal R(u) \leq C\}$ are strongly sequentially compact.
\end{itemize}
Then $u_n \to \bar u$ strongly in $U$.
\end{theorem}

The assumptions on the regulariser in Theorem~\ref{thm-convergence} are rather standard. 
Conditions of Theorem~\ref{thm-convergence} are satisfied, for example, if $\mathcal U=L_1$ and $\mathcal R(u) = \norm{u}_1 + TV(u)$. The term $\norm{u}_1$ can be dropped if boundedness of the $L_1$-norm can be guaranteed for all $u\in U$~\eqref{U}. The term $TV(u)$ can be replaced by any topologically equivalent seminorm, such as $TGV^2(u)$~\cite{bredies2009tgv} (see~\cite{l1tgv} for a proof of topological equivalence) or $TVL^p(u)$~\cite{Burger_TVLp_2016}. 

Strong compactness of the sub-level sets of $\mathcal R$ can be replaced by weak compactness if $\mathcal R$ has the Radon-Riesz property, i.e. that for any sequence $v_n \in \mathcal U$ weak convergence $v_n \weakto v$ along with convergence of the values $\mathcal R(v_n) \to \mathcal R(v)$ implies strong convergence $v_n \to v$. With this modification, Theorem~\ref{thm-convergence} admits norms in reflexive Banach spaces as regularisers.

The constraint $u \geq 0$ in~\eqref{U} is important. It can be relaxed to $u \geq a$ for some $a \in \mathcal U$ (not necessarily $\geq 0$), with some modifications in the formulae~\cite{Kor_Yag_IP:2013}, provided that the exact solution $\bar u$ satisfies this constraint. If the exact solution may be unbounded from below, the method won't work.

Let us briefly discuss the inclusion of the partial-order-based feasible set $U$~\eqref{U} in the norm-based feasible set $U_{h,\delta}$~\eqref{U_h_delta}. Let us choose
\begin{equation}\label{A_h_u_delta_from_bounds}
A_{h_n} = \frac{A^u_n+A^l_n}{2}, \quad f_{\delta_n} = \frac{f^u_n+f^l_n}{2}, \quad h_n=\frac{\norm{A^u_n-A^l_n}}{2}, \quad \delta_n = \frac{\norm{f^u_n-f^l_n}}{2}.
\end{equation}

\noindent It is easy to verify that $\norm{A-A_{h_n}} \leq h_n$, $\norm{f-f_{\delta_n}} \leq \delta_n$ and $(h_n,\delta_n) \to 0$ as $n \to \infty$. Indeed,  we note that, since $\forall n \,\, A^l_n \leq A \leq A^u_n$, we get
\begin{equation*}
-\frac{A^u_n-A^l_n}{2} \leq \frac{A^u_n+A^l_n}{2} - A \leq \frac{A^u_n-A^l_n}{2}
\end{equation*}
\noindent and therefore
\begin{equation} \label{modulus_A_h-A}
\left|\frac{A^u_n+A^l_n}{2} - A\right| \leq \frac{A^u_n-A^l_n}{2}.
\end{equation}
\noindent Since the space of regular operators $\mathcal U \to \mathcal F$ with $\mathcal F$ order complete is a Banach lattice under the so-called $r$-norm $\norm{A}_r = \norm{|A|}$~\cite{Abramovich} and the $r$-norm is always greater or equal to the operator norm~\cite{Schaefer},~\eqref{modulus_A_h-A} implies
\begin{equation*}
\norm{A_{h_n}-A} \leq \norm{A_{h_n}-A}_r \leq \frac{\norm{A^u_n-A^l_n}_r}{2} = \frac{\norm{A^u_n-A^l_n}}{2} = h_n.
\end{equation*}

The inequality $\norm{f-f_{\delta_n}} \leq \delta_n$ can be shown analogously and $(h_n,\delta_n) \to 0$ follows from~\eqref{bounds_sequences}. The proof of the inclusion of~\eqref{U} in~\eqref{U_h_delta} with $A_h,u_\delta,h,\delta$ as defined in~\eqref{A_h_u_delta_from_bounds} can be found in~\cite[Thm. 2]{Kor_IP:2014}. Therefore, the partial-order-based feasible set~\eqref{U} is contained in the norm-based feasible set~\eqref{U_h_delta} if the approximate operator, the approximate right-hand side and the approximation errors are as in~\eqref{A_h_u_delta_from_bounds}.

\section{Equivalence of $U$ and $U^*$}\label{section-U=U*}
Theorem~\ref{thm-convergence} guarantees convergence of approximate solutions chosen from the partial-order-based feasible set~$U$~\eqref{U} by minimizing a regulariser over $U$ as in~\eqref{optimisation_problem}. It is not clear, however, whether the minimisers  solve~\eqref{Au=f} for any particular pair $(A,f)$ within the bounds~\eqref{bounds}. In this section, we give a positive answer to this question for regular integral operators~\cite{Abramovich} acting between two spaces $\mathcal E$ and $\mathcal F$ of $(\mathcal S, \Sigma_1, \mu)$- and $(\mathcal T, \Sigma_2, \nu)$-measurable functions, respectively, where $\mathcal S$ and $\mathcal T$ are sets, $\Sigma_1$ and $\Sigma_2$ are $\sigma$-algebras over these sets and $\mu$ and $\nu$ are measures. A linear operator $A \colon \mathcal E \to \mathcal F$ is called an integral operator, if there exists a jointly measurable function $K(\cdot,\cdot)$ such that for each $u \in \mathcal E$ we have
\begin{equation*}
Au(t) = \int_{\mathcal S} K(s,t) u(s)  \, d\mu(s)
\end{equation*}
\noindent for $\nu$-almost all $t \in \mathcal T$. An integral operator $A \colon \mathcal E \to \mathcal F$ is regular if and only if the operator
\begin{equation*}
|A|u(t) \defeq \int_{\mathcal S} |K(s,t)| u(s)  \, d\mu(s)
\end{equation*}
\noindent has range in $\mathcal F$~\cite[Thm. 5.11]{Abramovich}.       

\begin{theorem}\label{thm-U_in_U*}
Let $\mathcal U$ and $\mathcal F$ in~\eqref{Au=f} be spaces of $(\mathcal S, \Sigma_1, \mu)$- and $(\mathcal T, \Sigma_2, \nu)$-measurable functions, respectively. Let $A^l$, $A^u$ be regular integral operators, $A^l \leq A^u$ and let $U$ be as defined in~\eqref{U}. Then for every $u \in U$ there exist a regular operator $A$, $A^l \leq A \leq A^u$, and $f \in \mathcal F$, $f^l \leq f \leq f^u$, such that $Au=f$.
\end{theorem}
\begin{proof}
Since $A^l$ and $A^u$ are integral operators, there exist jointly measurable functions $K^l(\cdot,\cdot)$ and $K^u(\cdot,\cdot)$ such that
\begin{equation*}
A^lu(t) = \int_{\mathcal S} K^l(s,t) u(s)  \, d\mu(s) \quad \text{and} \quad A^u u(t) = \int_{\mathcal S} K^u(s,t) u(s)  \, d\mu(s)
\end{equation*}
\noindent for $\nu$-almost all $t \in \mathcal T$ and by~\cite[Thm. 5.5]{Abramovich} we have $K^l(s,t) \leq K^u(s,t)$ for $\mu \times \nu$-almost all $(s,t) \in \mathcal S \times \mathcal T$.

Note that, as an immediate consequence of~\cite[Thm.~5.9]{Abramovich}, every operator $A$ that satisfies $A^l \leq A \leq A^u$ is an integral operator and therefore there exists a jointly measurable function $K(\cdot,\cdot)$ such that
\begin{equation*}
Au(t) = \int_{\mathcal S} K(s,t) u(s)  \, d\mu(s)
\end{equation*}
\noindent and
\begin{equation*}
K^l(s,t) \leq K(s,t) \leq K^u(s,t)
\end{equation*}
\noindent for $\mu \times \nu$-almost all $(s,t) \in \mathcal S \times \mathcal T$.

Let us choose a $\nu$-measurable function $\alpha(\cdot)$ such that $0 \leq \alpha(\cdot) \leq 1$ $\nu$-a.e. and define
\begin{equation*}
K(s,t) = (1-\alpha(t))K^l(s,t) + \alpha(t) K^u(s,t).
\end{equation*}

Note that such choice of $\alpha(t)$ does not capture all measurable functions $K(s,t)$ such that $K^l(s,t) \leq K(s,t) \leq K^u(s,t)$ (a choice of a jointly measurable $\alpha(s,t)$ would do that), but it will suffice for our existence proof.
Obviously, such choice of $K(s,t)$ defines an integral operator $A$ that satisfies $A^l \leq A \leq A^u$.

Fix $u^* \in U$ and define 
\begin{equation*}
f \defeq Au^* = \int_{\mathcal S} K(s,t) u^*(s) \, d\mu(s).
\end{equation*}
Our goal is to find $\alpha(\cdot)$ such that $0 \leq \alpha(\cdot) \leq 1$ and $f^l \leq f \leq f^u$, i.e.
\begin{equation*}
\left\{
\begin{aligned}
& f^l(t) \leq \int_{\mathcal S} [(1-\alpha(t))K^l(s,t) + \alpha(t) K^u(s,t)] u^*(s) \, d\mu(s) \leq f^u(t), \\
& 0 \leq \alpha(t) \leq 1
\end{aligned}
\right.
\end{equation*}
\noindent for $\nu$-almost all $t \in \mathcal T$. If $K^l(s,t) = K^u(s,t)$ $\mu$-almost everywhere, the inequalities above are trivially satisfied. Otherwise, we rewrite them as follows:
\begin{equation*}
\left\{
\begin{aligned}
 \frac{f^l(t) - \int_{\mathcal S} K^l(s,t) u^*(s) \, d\mu(s)}{\int_{\mathcal S} [K^u(s,t) -  K^l(s,t)] u^*(s) \, d\mu(s)}  &\leq 
& \alpha(t) &\leq 
\frac{f^u(t) - \int_{\mathcal S} K^l(s,t) u^*(s) \, d\mu(s)}{\int_{\mathcal S} [K^u(s,t) -  K^l(s,t)] u^*(s) \, d\mu(s)}, \\
 0 &\leq & \alpha(t) &\leq 1
\end{aligned}
\right.
\end{equation*}
\noindent or
\begin{equation}\label{ineq_alpha}
\begin{aligned}
\sup\left\{\frac{f^l(t) - \int_{\mathcal S} K^l(s,t) u^*(s) \, d\mu(s)}{\int_{\mathcal S} [K^u(s,t) -  K^l(s,t)] u^*(s) \, d\mu(s)},0\right\}  \leq 
 \alpha(t) \\
 \leq \inf\left\{\frac{f^u(t) - \int_{\mathcal S} K^l(s,t) u^*(s) \, d\mu(s)}{\int_{\mathcal S} [K^u(s,t) -  K^l(s,t)] u^*(s) \, d\mu(s)},1\right\}.
\end{aligned}
\end{equation}

This system has a solution if the first operand in the supremum in~\eqref{ineq_alpha}  is $\leq 1$ and the first operand in the infimum is $\geq 0$ $\nu$-a.e.~in $\mathcal T$. This is indeed the case as a consequence of the conditions~\eqref{bounds} and the inequalities $f^l \leq f^u$ and $A^l \leq A^u$. Therefore, we can always find a measurable $\alpha(\cdot)$ satisfying~\eqref{ineq_alpha}, for example, by choosing 
\begin{equation*}
\alpha(t) = \sup\left\{\frac{f^l(t) - \int_{\mathcal S} K^l(s,t) u^*(s) \, d\mu(s)}{\int_{\mathcal S} [K^u(s,t) -  K^l(s,t)] u^*(s) \, d\mu(s)},0\right\}, 
\end{equation*}
\noindent which is a supremum of two measurable functions and therefore measurable. In the special case $f^l = f^u = f$ we get a unique solution
\begin{equation*}
\alpha(t) =
\frac{f(t) - \int_{\mathcal S} K^l(s,t) u^*(s) \, d\mu(s)}{\int_{\mathcal S} [K^u(s,t) -  K^l(s,t)] u^*(s) \, d\mu(s)} \in [0,1] \text{ a.e. in $\mathcal T$}.
\end{equation*}
Hence, we have found a pair $(A,f)$ within the bounds~\eqref{bounds} for an arbitrary $u^* \in U$ such that $A u^* = f$.
\end{proof}

Theorem~\ref{thm-U_in_U*} proves the inclusion $U \subseteq U^*$, where $U^*$ is as defined in~\eqref{U*}. The opposite inclusion holds as well, since for any $u \in U^*$, with the corresponding pair $(A,f)$ from $U^*$ we have
\begin{equation*}
\left\{
\begin{aligned}
f^l &\leq f= Au \leq f^u, \\
A^lu &\leq f = Au \leq A^u u
\end{aligned}
\right.
\end{equation*}
due to the positivity of $u$, and hence $A^u u \geq f^l$ and $A^l u \leq f^u$. Therefore, we have proven the following
\begin{theorem}\label{thm-U=U*}
Under the assumptions of Theorem~\ref{thm-U_in_U*}, the sets $U$ and $U^*$ defined in~\eqref{U} and~\eqref{U*}, respectively, coincide.
\end{theorem}

An immediate consequence of this result is the convexity of the set $U^*$, since the set $U$ is, obviously, convex. An advantage of the formulation~\eqref{U} is the ease of implementation in an optimisation algorithm. On the other hand, the formulation~\eqref{U*} allows to easily include \emph{a priori} information on the operator $A$ as additional constraints, cf.~\eqref{U**}. 

\section{Imposing further constraints on the operator}\label{section-U**}
It is a natural question to ask, whether under some additional constraints on $A$ the feasible set $U^*$~\eqref{U*} remains convex (in~$u$). In this section we answer this question negatively in the case when the additional constraint is linear. We restrict ourselves to the finite-dimensional case when $\mathcal U = \mathbb R^n$, $\mathcal F = \mathbb R^m$, and $A$ is an $m \times n$ matrix. Note that in the finite-dimensional case partial order in the space of regular operators coincides with the elementwise partial order for matrices, i.e. $A \leq B$ iff $a_{ij} \leq b_{ij} \,\, \forall i,j$. Without loss of generality, we also restrict ourselves the special case $f^l=f^u=f$.

Fix a pair $(v,g) \in \mathbb R^n \times \mathbb R^m$ such that $v \geq 0$ and
\begin{equation}\label{bounds_vg}
A^l v \leq g \leq A^u v 
\end{equation}
\noindent and consider the set
\begin{equation}\label{U**_fin}
\begin{aligned}
U^{**} = \{u \in \mathbb R^n \colon u \geq 0, \,\, \exists A \in \mathbb R^{m \times n}, \,\, A^l  \leq A \leq A^u, \,\, Av=g, \,\, Au=f\}.
\end{aligned}
\end{equation}

As noted in the introduction, the additional constraint $Av=g$ can be useful, for example, if the exact forward operator is a convolution operator, i.e. all rows of the matrix $A$ sum up to one. This additional constraint allows us to further restrict the feasible set, while still preserving the inclusion $\bar u \in U^{**}$. Intuitively, a tighter feasible set provides more information about the exact solution and can be expected to improve the reconstructions. 

While the inclusion of $U^{**}$~\eqref{U**_fin} in the convex set $U$~\eqref{U}, obviously, still holds, the opposite inclusion does not hold any more. In what follows, we derive an explicit description of the set $U^{**}$ and argue that this set is not convex. Therefore, the advantages in reconstruction quality offered by using a tighter feasible set come at a price of a significant increase in computational complexity.

\paragraph{The structure of $U^{**}$} Every matrix $A$, $A^l  \leq A \leq A^u$, can be written as
\begin{equation*}
a_{i,j} = (1-\alpha_{i,j})a^l_{i,j} + \alpha_{i,j} a^u_{i,j}
\end{equation*}
\noindent with $\alpha_{i,j} \in [0,1]$. Fix $u \in U$. The constraints $Au=f$ and $Av=g$ can be written as
\begin{equation*}
\left\{
\begin{aligned}
\sum_{j=1}^n ((1-\alpha_{i,j})a^l_{i,j} + \alpha_{i,j} a^u_{i,j}) u_{i,j} = f_i, \\
\sum_{j=1}^n ((1-\alpha_{i,j})a^l_{i,j} + \alpha_{i,j} a^u_{i,j}) v_{i,j} = g_i \\
\end{aligned}
\right.
\end{equation*}
\noindent for each row $i = 1, \dots, m$. In what follows we will drop the subscript $i$ and consider this system for each row separately:
\begin{equation*}
\left\{
\begin{aligned}
\sum_{j=1}^n ((1-\alpha_j) a^l_j + \alpha_j a^u_j) u_j = f, \\
\sum_{j=1}^n ((1-\alpha_j)a^l_j + \alpha_j a^u_j) v_j = g. \\
\end{aligned}
\right.
\end{equation*}
\noindent or, equivalently,
\begin{equation}\label{system_alpha}
\begin{pmatrix}
(a^u_1-a^l_1)u_1 & \cdots & (a^u_n-a^l_n)u_n  \\
(a^u_1-a^l_1)v_1 & \cdots & (a^u_n-a^l_n)v_n  \\
\end{pmatrix}
\begin{pmatrix}
\alpha_1 \\
\vdots \\
\alpha_n
\end{pmatrix}
= 
\begin{pmatrix}
f - \sum_{j=1}^n a^l_j u_j \\
g - \sum_{j=1}^n a^l_j v_j
\end{pmatrix}.
\end{equation}

The matrix and the right-hand side in~\eqref{system_alpha} have non-negative entries due to~\eqref{U}, \eqref{bounds_vg} and the inequality $A^u \geq A^l$. Our goal is to find conditions on $u$ under which this system has a solution $\alpha \in [0,1]$. We will use  Farkas' lemma~\cite{Rockafellar} to find out, when the system~\eqref{system_alpha} has a positive solution. To find out, when it has a solution $\leq 1$, we reformulate~\eqref{system_alpha} in terms of $\beta = 1-\alpha$, which gives us the following system
\begin{equation}\label{system_beta}
\begin{pmatrix}
(a^u_1-a^l_1)u_1 & \cdots & (a^u_n-a^l_n)u_n  \\
(a^u_1-a^l_1)v_1 & \cdots & (a^u_n-a^l_n)v_n  \\
\end{pmatrix}
\begin{pmatrix}
\beta_1 \\
\vdots \\
\beta_n
\end{pmatrix}
= 
\begin{pmatrix}
\sum_{j=1}^n a^u_j u_j - f \\
\sum_{j=1}^n a^u_j v_j - g
\end{pmatrix},
\end{equation}
\noindent which also has a positive right-hand side due to~\eqref{U} and~\eqref{bounds_vg}. Combining these two systems, we get:
\begin{equation}\label{system_alpha_beta}
\begin{aligned}
\begin{pmatrix}
(a^u_1-a^l_1)u_1 & \cdots & (a^u_n-a^l_n)u_n  & 0 & \cdots & 0  \\
0 & \cdots & 0 & (a^u_1-a^l_1)u_1 & \cdots & (a^u_n-a^l_n)u_n  \\
(a^u_1-a^l_1)v_1 & \cdots & (a^u_n-a^l_n)v_n & 0 & \cdots & 0 \\
0 & \cdots & 0 & (a^u_1-a^l_1)v_1 & \cdots & (a^u_n-a^l_n)v_n \\
1 & \cdots & 0 & 1 & \cdots & 0 \\
\vdots & \ddots & \vdots & \vdots & \ddots & \vdots \\
0 & \cdots & 1 & 0 & \cdots & 1 
\end{pmatrix}
\begin{pmatrix}
\alpha_1 \\
\vdots \\
\alpha_n \\
\beta_1 \\
\vdots \\
\beta_n
\end{pmatrix}
\\ = 
\begin{pmatrix}
f - \sum_{j=1}^n a^l_j u_j \\
\sum_{j=1}^n a^u_j u_j - f \\
g - \sum_{j=1}^n a^l_j v_j \\
\sum_{j=1}^n a^u_j v_j - g \\
1 \\
\vdots \\
1
\end{pmatrix}.
\end{aligned}
\end{equation}

  The last $n$ lines in this system enforce the constraint $\beta = 1 - \alpha$, which guarantees that we find conditions under which the system~\eqref{system_alpha} has a solution that is simultaneously $\geq 0$ and $\leq 1$. The system~\eqref{system_alpha} has a solution in $[0,1]$ if and only if the system~\eqref{system_alpha_beta} has a solution $\geq 0$.

Our goal is to find the conditions on $u$, under which the system~\eqref{system_alpha_beta} has a non-negative solution. Farkas' lemma gives us the following alternative: either~\eqref{system_alpha_beta} has a solution $\geq 0$ or there exists a vector $y = (y_1, \cdots, y_{n+4})$ such that
\begin{equation}\label{A^Ty>=0}
\begin{aligned}
\begin{pmatrix}
(a^u_1-a^l_1)u_1 & 0 & (a^u_1-a^l_1)v_1  & 0  & 1 & \cdots & 0  \\
\vdots & \vdots & \vdots & \vdots & \vdots & \vdots & \vdots \\
(a^u_n-a^l_n)u_n & 0 & (a^u_n-a^l_n)v_n  & 0  & 0 & \cdots & 1  \\
0 & (a^u_1-a^l_1)u_1 & 0  & (a^u_1-a^l_1)v_1  & 1 & \cdots & 0  \\
\vdots & \vdots & \vdots & \vdots & \vdots & \vdots & \vdots \\
0 & (a^u_n-a^l_n)u_n & 0  & (a^u_n-a^l_n)v_n  & 0 & \cdots & 1  \\
\end{pmatrix}
\begin{pmatrix}
y_1 \\
\vdots \\
y_{n+4}
\end{pmatrix}
 \geq 
\begin{pmatrix}
0 \\
\vdots \\
0
\end{pmatrix}
\end{aligned}
\end{equation}
\noindent and
\begin{equation}\label{(y,b)<0}
\begin{aligned}
y_1 \left(f - \sum_{j=1}^n a^l_j u_j\right) + y_2 \left(\sum_{j=1}^n a^u_j u_j - f\right) + y_3 \left(g - \sum_{j=1}^n a^l_j v_j\right) \\
+ y_4 \left(\sum_{j=1}^n a^u_j v_j - g\right) +
\sum_{j=1}^n y_{j+4} < 0.
\end{aligned}
\end{equation}

We can rewrite these conditions equivalently as follows:
\begin{eqnarray}
(a^u_j-a^l_j) (u_j y_1 + v_j y_3) + y_{j+4} \geq 0, \quad j=1,\dots,n, \label{eq_1} \\ 
(a^u_j-a^l_j) (u_j y_2 + v_j y_4) + y_{j+4} \geq 0, \quad j=1,\dots,n, \label{eq_2} \\ 
\left(f- \sum_{j=1}^n a^l_j u_j\right)y_1 + \left(\sum_{j=1}^n a^u_j u_j - f\right)y_2 + \left(g-\sum_{j=1}^n a^l_j v_j\right) y_3 \nonumber \\ 
 + \left(\sum_{j=1}^n a^u_j v_j - g\right)y_4 + \sum_{j=1}^n y_{j+4} < 0. \label{eq_3}
\end{eqnarray}

The proof will be based on considering various combinations of signs of $(y_1-y_2)$ and $(y_3-y_4)$ separately. First we prove the following
\begin{lemma}\label{lemma}
If a solution of the system~\eqref{eq_1}--\eqref{eq_3} exists, it satisfies the inequality $(y_1-y_2)(y_3-y_4)<0$.
\end{lemma}
\begin{proof}
Summing up equations~\eqref{eq_1} and~\eqref{eq_2}, we get the following system:
\begin{eqnarray*}
y_1 \sum_{j=1}^n (a^u_j-a^l_j) u_j + y_3 \sum_{j=1}^n (a^u_j-a^l_j) v_j  + \sum_{j=1}^n y_{j+4} \geq 0,  \\
y_2 \sum_{j=1}^n (a^u_j-a^l_j) u_j + y_4 \sum_{j=1}^n (a^u_j-a^l_j) v_j  + \sum_{j=1}^n y_{j+4} \geq 0,  \\
\left(f- \sum_{j=1}^n a^l_j u_j\right)y_1 + \left(\sum_{j=1}^n a^u_j u_j - f\right)y_2 + \left(g-\sum_{j=1}^n a^l_j v_j\right) y_3 \\ 
 + \left(\sum_{j=1}^n a^u_j v_j - g\right)y_4 + \sum_{j=1}^n y_{j+4} < 0,
\end{eqnarray*}
\noindent which implies 
\begin{eqnarray*}
(y_1 - y_2) \left(\sum_{j=1}^n a^u_j u_j - f\right) + (y_3 - y_4) \left(\sum_{j=1}^n a^u_j v_j - g\right) > 0, \\
(y_1 - y_2) \left(f - \sum_{j=1}^n  a^l_j u_j\right) + (y_3 - y_4) \left( g - \sum_{j=1}^n a^l_j v_j\right) < 0.  \\
\end{eqnarray*}

\noindent The coefficients at $(y_1 - y_2)$ and $(y_3 - y_4)$ in both equations are positive. Therefore, whether $y_1 \leq y_2 \land y_3 \leq y_4$ or $y_1 \geq y_2 \land y_3 \geq y_4$, one of the above equations is violated.
\end{proof}

Due to Lemma~\ref{lemma}, we only need to consider two combinations: $y_1 > y_2 \land y_3<y_4$ and $y_1 < y_2 \land y_3>y_4$. 

Let $y_1 > y_2 \land y_3<y_4$. Equations~\eqref{eq_1}--\eqref{eq_2} are equivalent to the following system:
\begin{equation*}
y_{j+4} \geq -(a^u_j-a^l_j) \min\{u_j y_1 + v_j y_3, u_j y_2 + v_j y_4\}, j=1,\dots,n.
\end{equation*}

Let $J = \{j \colon u_j y_1 + v_j y_3 \leq u_j y_2 + v_j y_4\} = \{j \colon u_j  \leq v_j  \frac{y_4 - y_3}{y_1-y_2}\}$ and $J_c$ be the complement of $J$. Inequality~\eqref{eq_3} requires that we choose $y_{j+4}$ as small as possible, which is
\begin{equation*}
y_{j+4} = -(a^u_j-a^l_j)\left\{
\begin{aligned}
&u_j y_1 + v_j y_3, \quad j \in J,\\
&u_j y_2 + v_j y_4, \quad j \in J_c.
\end{aligned} 
\right.
\end{equation*}
Substituting this into~\eqref{eq_3}, we get
\begin{equation}\label{eq_4}
\begin{aligned}
(y_1 - y_2) \left[ \sum_{j \in J} a^u_j v_j \left( \frac{y_4-y_3}{y_1-y_2}-\frac{u_j}{v_j} \right)
+ \sum_{j \in J_c} a^l_j v_j \left( \frac{y_4-y_3}{y_1-y_2}-\frac{u_j}{v_j}  \right) \right.\\
\left. + f - g\frac{y_4-y_3}{y_1-y_2} \right] < 0.
\end{aligned}
\end{equation}
Define $z \defeq \frac{y_4-y_3}{y_1-y_2} > 0$ and
\begin{equation}\label{phi(z)}
\begin{aligned}
\varphi(z) \defeq \sum_{j \in J} a^u_j v_j \left( z-\frac{u_j}{v_j} \right)
+ \sum_{j \in J_c} a^l_j v_j \left( z-\frac{u_j}{v_j}  \right) + f - gz   \\
= \sum_{j=1}^n \left(z-\frac{u_j}{v_j}\right) v_j 
\left\{
\begin{aligned}
a^u_j, \quad  \frac{u_j}{v_j} \leq z,\\
a^l_j, \quad \frac{u_j}{v_j} \geq z
\end{aligned}
\right\}
+f -gz.
\end{aligned}
\end{equation}

\begin{lemma}
The function $\varphi(z)$ as defined in~\eqref{phi(z)} has the following properties:
\begin{enumerate}
\item $\varphi(z)$ is piecewise linear; \label{num_1}
\item $\varphi'(z) = \sum_{j=1}^n v_j
\left\{
\begin{aligned}
a^u_j, \quad  \frac{u_j}{v_j} < z,\\
a^l_j, \quad \frac{u_j}{v_j} > z
\end{aligned}
\right\} - g
$, $z \neq \frac{u_k}{v_k}$, $k = 1,\dots,n$; \label{num_2}
\item $\varphi'(z \leq \min_j \frac{u_j}{v_j}) = \sum_{j=1}^n a^l_j v_j - g \leq 0$; \label{num_3}
\item $\varphi'(z \geq \max_j \frac{u_j}{v_j}) = \sum_{j=1}^n a^u_j v_j - g \geq 0$; \label{num_4}
\item $\varphi'(z)$ is monotonically non-decreasing; \label{num_5}
\item $\varphi(z)$ is continuous; \label{num_6}
\item $\varphi(z)$ is convex on $(0,\infty)$.
\end{enumerate}
\end{lemma}
\begin{proof}
\begin{enumerate}
\item[1.--4.] Obvious.
\item[5.] Every time $z$ crosses a point $\frac{u_j}{v_j}$ from left to right, one $a^l_j$ is replaced by a greater value $a^u_j$, and between these points $\varphi'(z)$ is constant, hence the monotonicity of $\varphi'(z)$.
\item[6.] Suspected jumps at $z = \frac{u_j}{v_j}$ are zero, since the summand in~\eqref{phi(z)} at $j \colon z = \frac{u_j}{v_j}$ is zero.
\item[7.] Follows from the above.
\end{enumerate}
\end{proof}

The subdifferential of $\varphi(z)$ is given by:
\begin{equation*}
\subdiff \varphi(z) = 
\left\{
\begin{aligned}
&\sum_{j=1}^n v_j
\left\{
\begin{aligned}
a^u_j, \quad  \frac{u_j}{v_j} < z,\\
a^l_j, \quad \frac{u_j}{v_j} > z
\end{aligned}
\right\} - g, \,\, z \neq \frac{u_k}{v_k}, \,\, k = 1,\dots,n, \\
&\left[\varphi'\left(\frac{u_k}{v_k}-0\right), \varphi'\left(\frac{u_k}{v_k}+0\right)\right], \quad z = \frac{u_k}{v_k}, \,\, k = 1,\dots,n.
\end{aligned}
\right.
\end{equation*}
The minimum of $\varphi(z)$ is obtained at a point $z=\frac{u_{k^*}}{v_{k^*}}$ such that $0 \in \subdiff \varphi \left(\frac{u_{k^*}}{v_{k^*}}\right)$. Let us permute the indices so that $\frac{u_k}{v_k}$ are sorted in ascending order (the entries in $a^u_k$ and $a^l_k$ must be re-sorted accordingly). This operation does not affect the products of $a^l$, $a^u$ with $u$ and $v$. Then $k^*$ is given by the following condition:
\begin{equation}\label{condition_k*_phi}
\sum_{j=1}^{k^*-1} a^u_j v_j + \sum_{j=k^*}^{n} a^l_j v_j -g  \leq 0 
\leq \sum_{j=1}^{k^*} a^u_j v_j + \sum_{j=k^*+1}^{n} a^l_j v_j - g
\end{equation}
\noindent and the minimum of $\varphi(z)$ is 
\begin{equation}\label{phi_min}
\varphi_{min} = \varphi\left(\frac{u_{k^*}}{v_{k^*}}\right) = 
\sum_{j=1}^n \left(\frac{u_{k^*}}{v_{k^*}}-\frac{u_j}{v_j}\right) v_j 
\left\{
\begin{aligned}
a^u_j, \quad  \frac{u_j}{v_j} \leq \frac{u_{k^*}}{v_{k^*}},\\
a^l_j, \quad \frac{u_j}{v_j} \geq \frac{u_{k^*}}{v_{k^*}}
\end{aligned}
\right\}
+f -g\frac{u_{k^*}}{v_{k^*}}.
\end{equation}
Note that, although the condition for $k^*$~\eqref{condition_k*_phi} does not contain $u$, $k^*$ does depend on $u$ due to the permutation of the indices we made.

Conditions~\eqref{condition_k*_phi}-\eqref{phi_min} define the minimum of the function $\varphi(z)$. If this minimum is negative, then the system~\eqref{eq_1}-\eqref{eq_3} has a solution such that $y_1 > y_2 \land y_3 < y_4$ and the original  system~\eqref{system_alpha_beta} has no non-negative solution. In order for the system~\eqref{system_alpha_beta} to have a non-negative solution, we must have
\begin{equation}\label{phi_min>=0}
\varphi_{min}(u) =  
\sum_{j=1}^n \left(\frac{u_{k^*}}{v_{k^*}}-\frac{u_j}{v_j}\right) v_j 
\left\{
\begin{aligned}
a^u_j, \quad  \frac{u_j}{v_j} \leq \frac{u_{k^*}}{v_{k^*}},\\
a^l_j, \quad \frac{u_j}{v_j} \geq \frac{u_{k^*}}{v_{k^*}}
\end{aligned}
\right\}
+f -g\frac{u_{k^*}}{v_{k^*}} \geq 0.
\end{equation}

Proceeding similarly in the case $y_1 < y_2 \land y_3 > y_4$, we obtain the following conditions:
\begin{equation}\label{psi_min>=0}
\psi_{min}(u) \defeq   \sum_{j=1}^n \left(\frac{u_j}{v_j} - \frac{u_{k^{**}}}{v_{k^{**}}}\right) v_j 
\left\{
\begin{aligned}
a^u_j, \quad  \frac{u_j}{v_j} \geq \frac{u_{k^{**}}}{v_{k^{**}}},\\
a^l_j, \quad \frac{u_j}{v_j} \leq \frac{u_{k^{**}}}{v_{k^{**}}}
\end{aligned}
\right\}
+ g\frac{u_{k^{**}}}{v_{k^{**}}} - f \geq 0,
\end{equation}
\noindent where $k^{**}$ is defined by the following condition (the indices are  assumed again to be permuted in such a way  that $\frac{u_k}{v_k}$ are sorted in ascending order):
\begin{equation}\label{condition_k**_psi}
 \sum_{j=1}^{k^{**}} a^l_j v_j + \sum_{j=k^{**}+1}^{n} a^u_j v_j - g \leq 0 
\leq \sum_{j=1}^{k^{**}-1} a^l_j v_j + \sum_{j=k^{**}}^{n} a^u_j v_j -g.
\end{equation}
Note that $k^*$ and $k^{**}$ are, in general, different.

We have proven the following
\begin{theorem}\label{thm-U**}
The set $U^{**}$ defined in~\eqref{U**_fin} consists of all $u \in U$ (as defined in~\eqref{U*}) for which conditions~\eqref{condition_k*_phi}, \eqref{phi_min>=0} and \eqref{psi_min>=0}, \eqref{condition_k**_psi} are satisfied.
\end{theorem}

\begin{remark}\label{U**_not_convex}
If there was no dependence of $k^*$ and $k^{**}$ on $u$, the functions $\varphi_{min}(u)$ as in~\eqref{phi_min} and $\psi_{min}(u)$ as in~\eqref{psi_min>=0} would be convex and the set $U^{**}$ would be a difference of the convex set $U$ and two convex sets defined by the inequalities $\varphi_{min}(u)<0$ and $\psi_{min}(u)<0$. The dependence of $k^*$ and $k^{**}$ on $u$ makes the structure of $U^{**}$ more complicated. We do not study this structure further in this work.
\end{remark}

Figure~\ref{pic-U**} shows the set $U^{**}$ in the case when $\mathcal U=\mathbb R^2$, $\mathcal F = \mathbb R$, $A^l$ and $A^u$ are randomly chosen in the intervals $[0,1]^2$ and $[1,2]^2$, respectively, $f$ is generated using the matrix $(A^u+A^l)/2$ and a randomly chosen $u \in [0,25]^2$, and  $f^u = f^l = f$. To visualise the feasible set $U^{**}$, we pick random points $u$ in the positive quadrant and check conditions~\eqref{condition_k*_phi}, \eqref{phi_min>=0} and \eqref{psi_min>=0}, \eqref{condition_k**_psi} (points that satisfy these conditions are shown as blue stars in Fig.~\ref{pic-U**}). In this simple example, the set $U^{**}$ can be computed analytically as well, the result is shown by two solid lines in Fig.~\ref{pic-U**}. As expected, the result is the same.

\begin{figure}[t]
\centering
\includegraphics[height=0.5\textwidth]{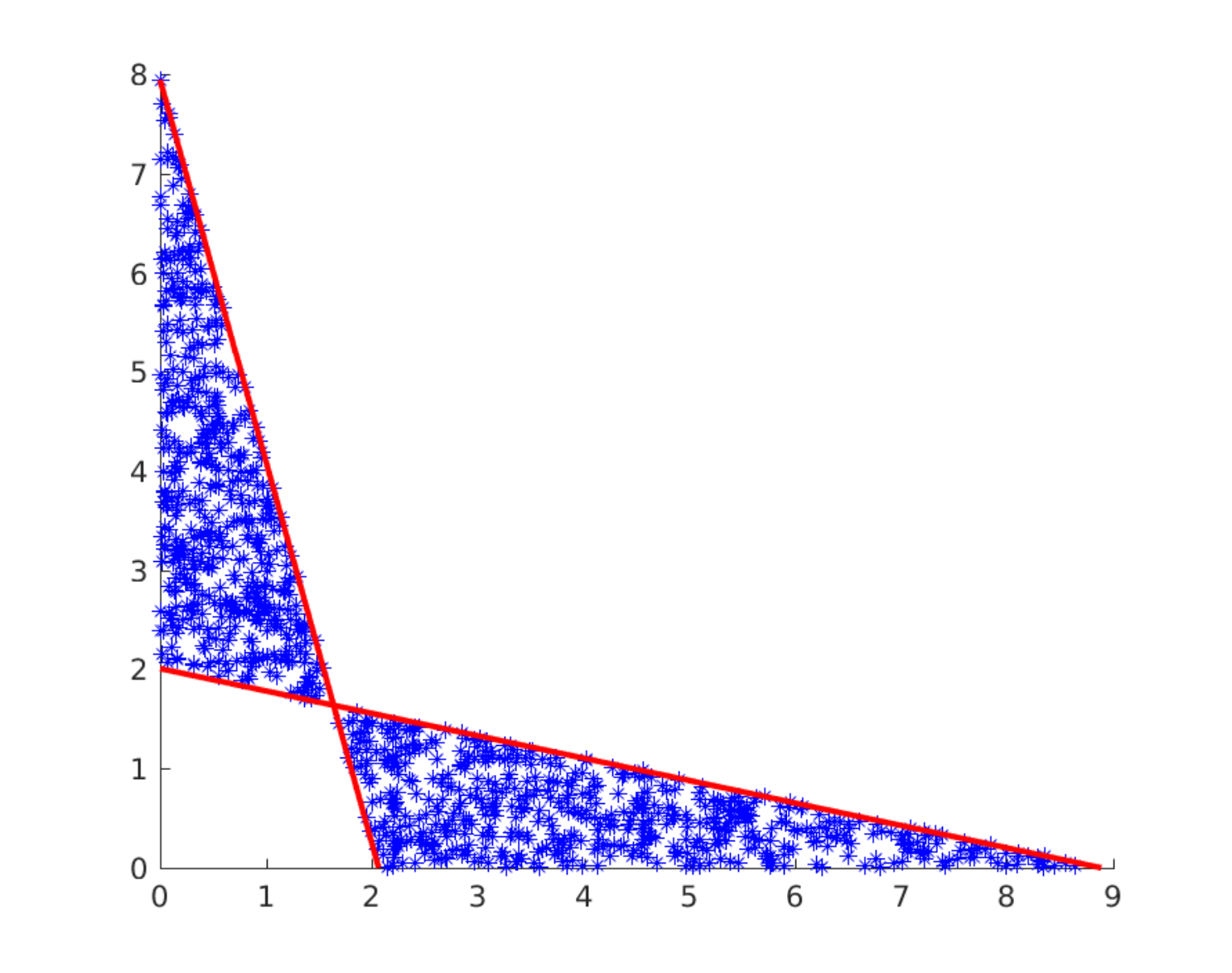}
\caption{The feasible set $U^{**}$ in a 2D toy problem generated using conditions~\eqref{condition_k*_phi}, \eqref{phi_min>=0} and \eqref{psi_min>=0}, \eqref{condition_k**_psi} (blue stars) and using analytic formulas (between red solid lines). As expected, the two sets coincide.} 
\label{pic-U**}
\end{figure}

The result of Theorem~\ref{thm-U**} gives the impression that the information that $A$ is a convolution matrix (which can be expressed as the condition $Ae=e$) is of little use for the approach, since including this information in the reconstruction algorithm requires solving a non-convex optimisation problem. However, the additional linear constraint can sometimes be used to tighten the bounds $A^l$, $A^u$, if they weren't carefully chosen initially. Indeed, one can attempt finding tighter lower and upper bounds by solving the following optimisation problems:
\begin{equation}\label{new_Al_Au}
\tilde a^l_{ij} = \min_{A^l \leq A \leq A^u, \,\, Ae = e} a_{ij}, \quad \tilde a^u_{ij} = \max_{A^l \leq A \leq A^u, \,\, Ae = e} a_{ij}.
\end{equation}
\noindent These optimisation problems are convex and can be efficiently solved in parallel.

In the infinite-dimensional case, the analogue of the optimisation problems~\eqref{new_Al_Au} is as follows:
\begin{equation}\label{new_Al_Au_infinite}
\tilde A^l = \inf \{A \colon A^l \leq A \leq A^u, \,\, Ae = e\}, \quad \tilde A^u = \sup \{A \colon A^l \leq A \leq A^u, \,\, Ae = e\},
\end{equation}
\noindent where the $\inf$ and $\sup$ are taken in the space of regular operators $\mathcal U \to \mathcal F$. For these $\inf$ and $\sup$ to exist, the space of regular operators $\mathcal U \to \mathcal F$ must be order complete, i.e. any majorised set in it must have a supremum. This is guaranteed when $\mathcal F$ is order complete~\cite[Theorem 1.16]{Abramovich}. For example, the spaces of measurable functions $L_p(\mathcal S, \Sigma, \mu)$ are order complete, whilst the space of continuous functions $C(\mathcal S)$ is not~\cite{Schaefer}.

\begin{remark}
An interesting question is, what is the convex hull $\conv U^{**}$ of $U^{**}$. If it is smaller than $U,$ then the feasible set for $u$ can be tightened while preserving its convexity and better reconstructions can be expected. It is clear that in some situations $\conv U^{**}$ is a strict subset of $U$, for example, if $A^l \neq \inf \{A \colon A^l \leq A \leq A^u, \,\, Ae = e\}$ or  $A^u \neq \sup \{A \colon A^l \leq A \leq A^u, \,\, Ae = e\}$ (cf.~\eqref{new_Al_Au_infinite}). However, it is hard to say anything more about $\conv U^{**}$ at the first glance. We leave the study of $\conv U^{**}$ for future work.
\end{remark}

\section{Applications in deblurring}\label{section-deblurring}
Deblurring is widely used to improve the quality of images, for example, in astronomy~\cite{deconvolution_astronomy_review} and fluorescence microscopy~\cite{deconvolution_microscopy_review, Bertero_et_al_deblurring_review:2009}. Quite often, we only have an estimate of the blurring kernel, for example, when it is measured experimentally~\cite{Telescope_PSF, Shaw_PSF_91} or obtained using simplified models~\cite{Zhang:07}. Therefore, it is necessary to account for the uncertainty in the blurring kernel during the reconstruction. One possibility is to estimate the kernel and the image simultaneously, which is known as blind deblurring~\cite{Chan_Wong_blind_deconvolution_98, Perrone_Favaro_blind_deconvolution_16}. The problem with this approach is that it results in non-convex optimisation problems and is severely ill-posed. Another option is to include the knowledge about the uncertainty in the blurring operator, if such knowledge is available, into the reconstruction process, which is the approach that we pursue. 

\paragraph{Deblurring in 1D} Let us first consider a simple one-dimensional example to get a feeling for how noise in the operator affects reconstruction and how the proposed approach can alleviate the impact of this noise. Consider the signal $u$ shown in Fig.~\ref{pic-deblurring_signal} in blue (dashed line). This signal is convolved with a Gaussian blurring kernel 
with standard deviation $0.5$ and Dirichlet boundary conditions. Then uniform noise with support $[-c,c]$ with $c=0.005\max_i |u_i|$ is added to it (i.e. we add $1\%$ noise). The blurred and noisy signal is shown in Fig.~\ref{pic-deblurring_signal} in green (solid line). 
The ground-truth signal is piecewise-constant, suggesting the use of total variation~\cite{ROF} as the regulariser.

First we reconstruct the signal using the exact operator $A$ by solving the optimisation problem~\eqref{optimisation_problem} with $\R(u) = \TV(u)$, $A^l = A^u = A$. All reconstructions were computed using CVX~\cite{cvx,cvx2}. The bounds for the right-hand side $f^u$ and $f^l$ can be obtained from the noisy signal $f$ as follows: $f^l = f-c$, $f^u = f+c$. Note that in this case the problem~\eqref{optimisation_problem} is equivalent to the following one:
\begin{equation}\label{residual_method_L_infty_exact}
\min \TV(u) \quad \text{ s.t. } u \geq 0, \,\, \norm{Au-f}_\infty \leq c.
\end{equation}

Since $\TV$ is not strictly convex, we cannot guarantee uniqueness of the solution of~\eqref{residual_method_L_infty_exact}. Non-uniqueness of some $\TV$-based reconstruction models, e.g., the $\TV-L_1$ model, is a well-known issue~\cite{Chan_Esedoglu_TVL1}. In order to ensure uniqueness we add a small correction term $\gamma \norm{u}_2$ with $\gamma = 10^{-4}$ to the regulariser in~\eqref{residual_method_L_infty_exact}. In order to simplify notation we omit this correction term in the statements of optimisation problems involving $\TV$ that follow.

As expected, using the exact forward operator and a suitable regulariser we get a nearly perfect reconstruction (Fig.~\ref{pic-reconstruction_exact}).

\begin{figure}[t]
    \centering
    \begin{subfigure}[t]{0.45\textwidth}
	\includegraphics[width=\textwidth]{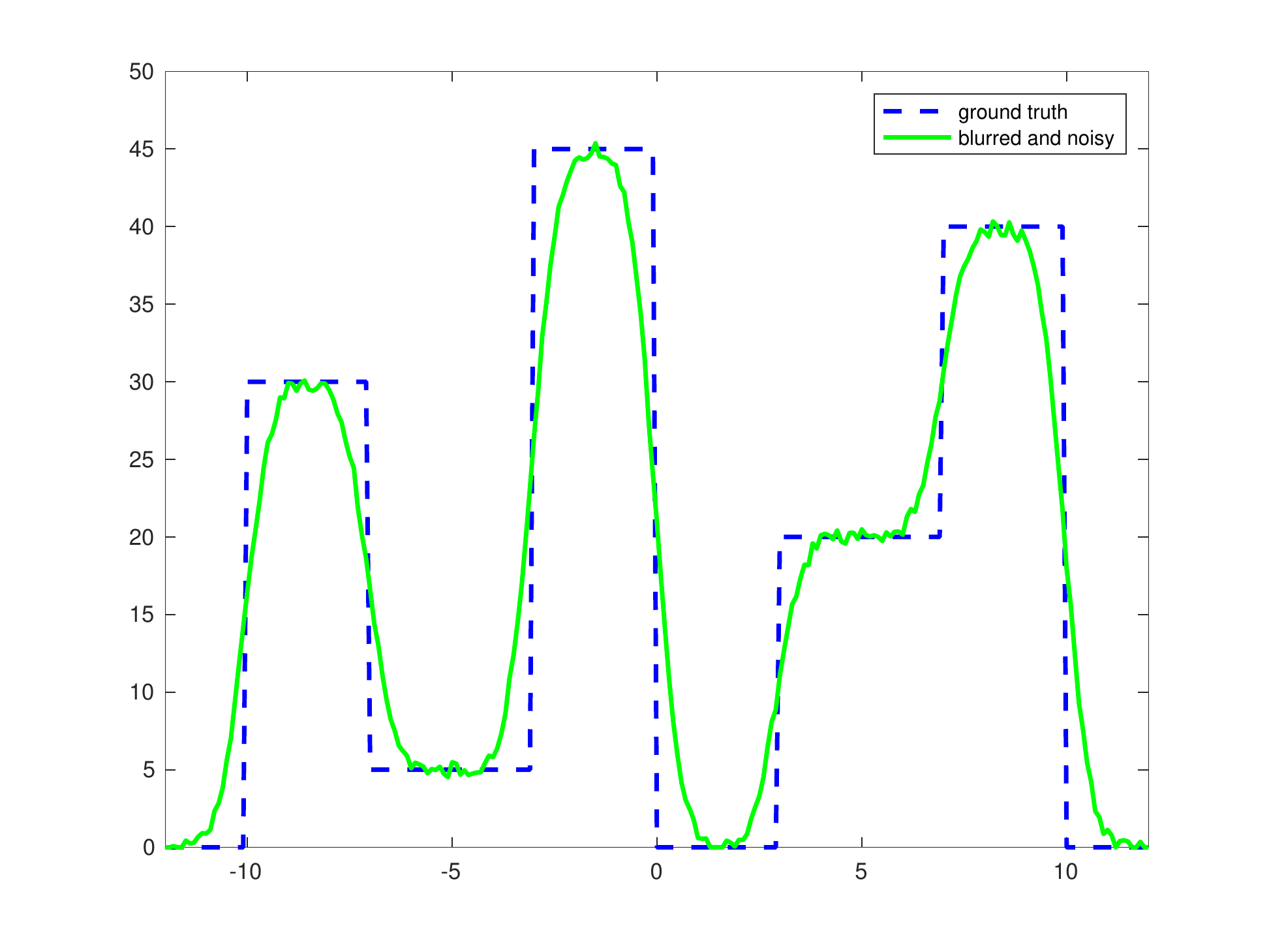}
	\subcaption{Ground truth (blue dashed line) and blurred and noisy signal (green solid line). $\PSNR = 18.3$, $\SSIM = 0.08$.}
	\label{pic-deblurring_signal}    
    \end{subfigure}
\begin{subfigure}[t]{0.45\textwidth}
            \centering
	\includegraphics[width=\textwidth]{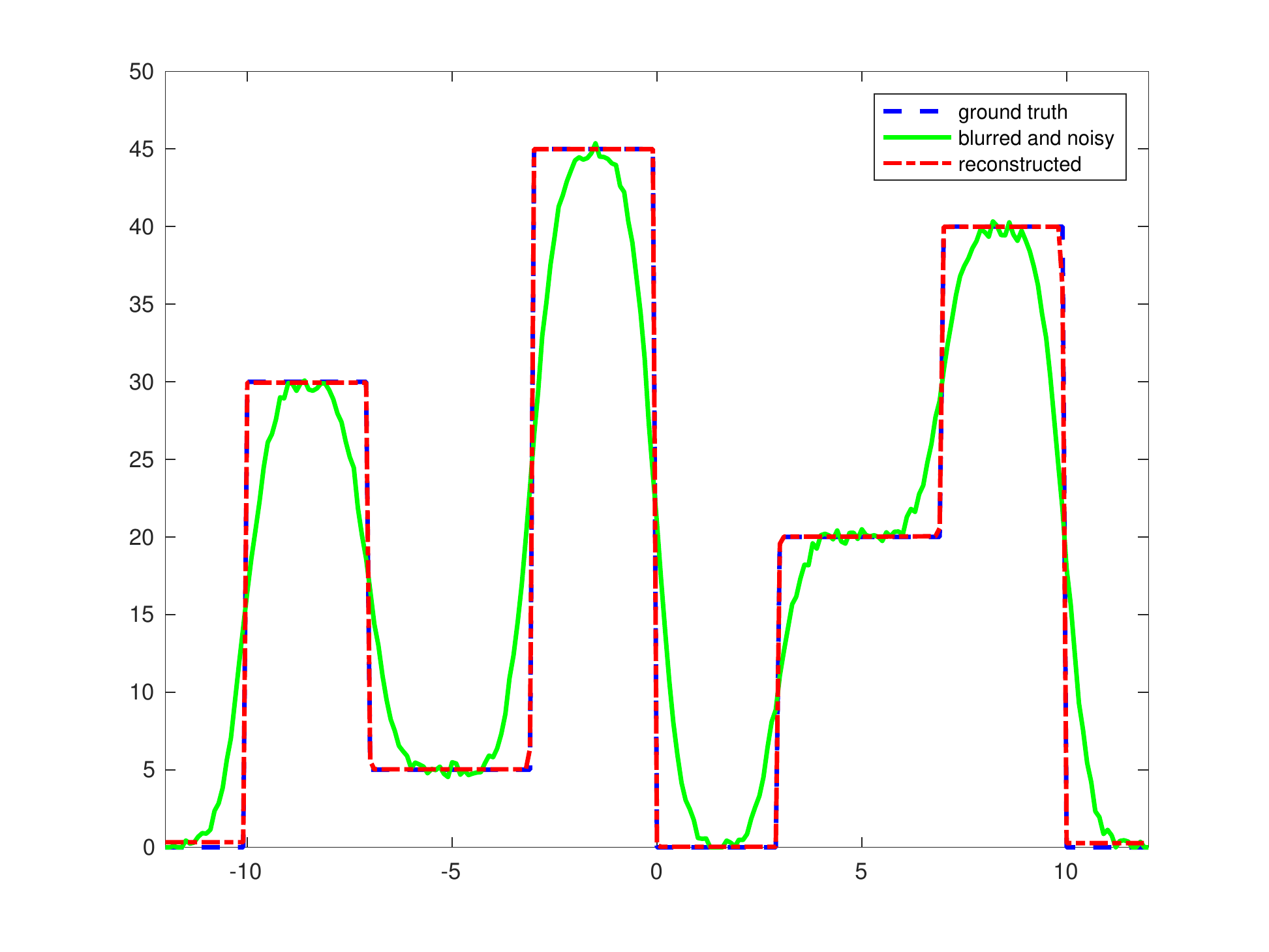}
	\subcaption{Reconstruction using the exact operator (red dash-dotted line). $\PSNR = 43.8$, $\SSIM = 0.77$.}
	\label{pic-reconstruction_exact}    \end{subfigure}
\\
\begin{subfigure}[t]{0.45\textwidth}
            \centering
	\includegraphics[width=\textwidth]{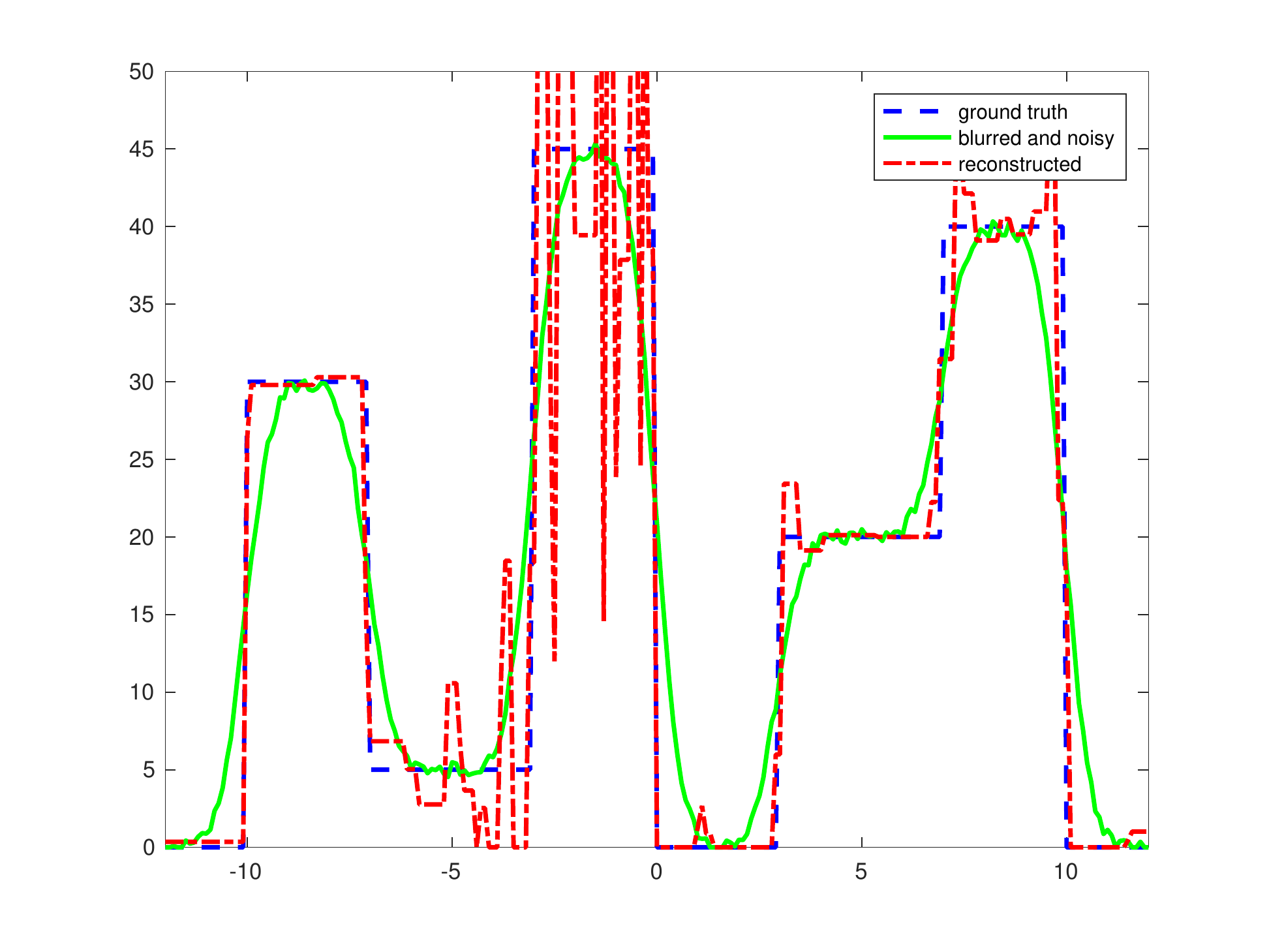}
	\subcaption{Reconstruction using a noisy operator (red dash-dotted line).  $\PSNR = 13.8$, $\SSIM = 0.22$. }
	\label{pic-reconstruction_wrong}
    \end{subfigure}
\begin{subfigure}[t]{0.45\textwidth}
            \centering
	\includegraphics[width=\textwidth]{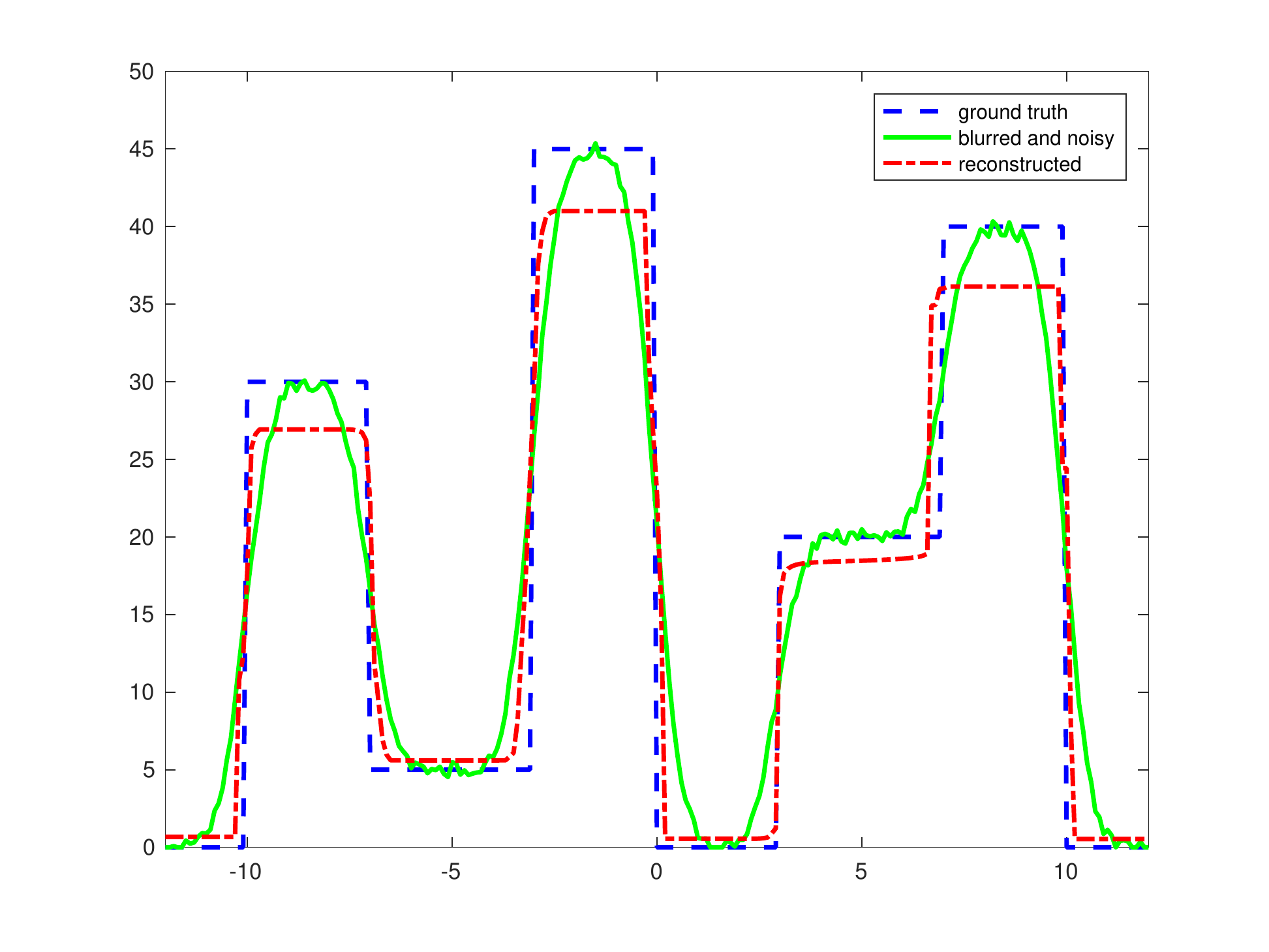}
	\subcaption{Reconstruction using interval bounds for the operator (red dash-dotted line). $\PSNR = 22.3$, $\SSIM = 0.61$.}
	\label{pic-reconstruction_bounds}
    \end{subfigure}
    \caption{Reconstruction of a piecewise constant signal using total variation. Perfect knowledge of the blurring operator yields nearly perfect reconstruction (\subref{pic-reconstruction_exact}). However, even $10\%$ noise in the blurring operator renders the inversion ill-posed (\subref{pic-reconstruction_wrong}). Taking uncertainty in the blurring operator into account yields stable reconstruction, but results in some loss of contrast (\subref{pic-reconstruction_bounds}).}
\end{figure}

Let us assume that only a slightly perturbed version $\tilde A$ of the blurring operator $A$ is available: 
\begin{equation}\label{noise_matrix}
\tilde a_{ij} = \max\{a_{ij} + r_{ij} \cdot d,0\},
\end{equation}
\noindent where $d=0.05*\max_{k,l} a_{kl}$ and $r_{ij}$ are i.i.d.~uniform random numbers with support $[-1,1]$ (i.e. we add $10\%$ noise to the operator).  Note that while the true blurring matrix $A$ is typically sparse, this is not any longer true for the perturbed matrix $\tilde A$. It reasonable to assume, however, that all entries smaller than $d$ are pure noise and set them to zero. We also take into account that the rows of the  blurring matrix must sum up to one and normalise the rows of $\tilde A$.

Let us solve the optimisation problem~\eqref{optimisation_problem} with $\R(u) = \TV(u)$, $A^l = A^u = \tilde A$ and $f^u, f^l$ as defined above. This is equivalent to solving
\begin{equation}\label{residual_method_L_infty_wrong}
\min \TV(u) \quad \text{ s.t. } u \geq 0, \,\, \norm{\tilde Au-f}_\infty \leq c.
\end{equation}

The condition that the exact solution $\bar u$ is in the feasible set $U$~\eqref{U}, which is crucial for the convergence proof (Theorem~\ref{thm-convergence}), does not hold in this case any more and therefore the regularising properties of the approach~\eqref{optimisation_problem} can not be guaranteed. Indeed, we observe  that an error in the operator of $10\%$ renders the inversion ill-posed and the reconstruction highly oscillatory (Fig.~\ref{pic-reconstruction_wrong}). This problem can be dealt with by increasing the allowed noise level in the problem~\eqref{residual_method_L_infty_wrong}, e.g., by multiplying  the right-hand side of~\eqref{residual_method_L_infty_wrong} by a factor $C>1$, however, the value of $C$ that will be sufficient to remove oscillations depends on the noise in the operator and is not straightforward to determine.

Let us now acknowledge the fact that the operator $\tilde A$ contains errors and derive lower and upper bounds for the unknown true operator from~\eqref{noise_matrix} as follows: 
\begin{equation}\label{bounds_matrix}
a^u_{ij} = \tilde a_{ij} + d, \quad a^l_{ij} = \max\{\tilde a_{ij} -d,0\} 
\end{equation}

\noindent with $d=0.05*\max_{k,l}\tilde a_{kl}$. Since we assumed that we know the support of the blurring kernel when we calculated $\tilde A$ (while setting all entries in $\tilde A$ below a certain threshold to zero), we will also use this information in determining $A^u$ and set $a^u_{ij}$ to zero whenever $\tilde a_{ij}=0$. 

For the reconstruction we  solve the following problem:
\begin{equation}\label{residual_bounds}
\min \TV(u) \quad \text{ s.t. } u \geq 0, \,\, A^l u \leq f^u, \,\, A^u u \geq f^l.
\end{equation}

The result is shown in Fig.~\ref{pic-reconstruction_bounds}. We observe that the oscillations disappear but there is some loss of contrast compared to the reconstruction using the exact operator (Fig~.~\ref{pic-reconstruction_exact}).
These results can be explained as follows. If the feasible set in~\eqref{U} is defined using a noisy operator, it does not contain, in general, the exact solution and therefore no stable approximation of the exact solution can be achieved using elements from this feasible set. Explicitly accounting for the uncertainty in the operator makes the feasible set larger and guarantees the inclusion of the exact solution, enabling stable reconstruction. However, this increased size of the feasible set is also responsible for the loss of contrast observed in Figure~\ref{pic-reconstruction_bounds} as compared to the reconstruction using the exact operator (Figure~\ref{pic-reconstruction_bounds}). The reason for this loss of contrast is that, given the freedom to choose solutions from a larger feasible set, $\TV$ seeks to find one with smallest possible contrast.

\paragraph{Deblurring in 2D}
Let us now turn to two-dimensional images. All images used in this section are greyscale images (with values between $0$ and $255$) of size $128\times128$ pixels. Consider first a piecewise constant image shown in Figure~\ref{pic-squares_orig}. This image is convolved with a Gaussian blur kernel with standard deviation $1$. Neumann boundary conditions are used. Then uniform noise with support $[-c;c]$ with $c=10$ is added to it, which corresponds to $8 \%$ noise. The blurred and noisy image is shown in Figure~\ref{pic-squares_noisy}.

\begin{figure}[t!]
    \centering
    \begin{subfigure}[t]{0.32\textwidth}
            \centering
            \includegraphics[width=\textwidth]{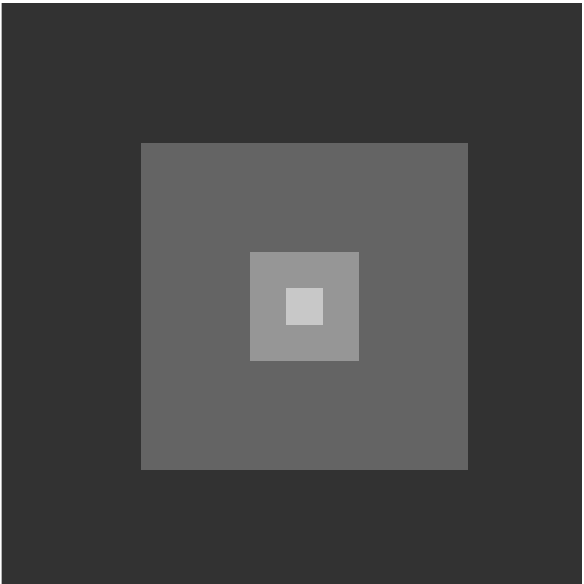}
            \subcaption{Original image}
    	\label{pic-squares_orig}
    \end{subfigure}
\begin{subfigure}[t]{0.32\textwidth}
            \centering
            \includegraphics[width=\textwidth]{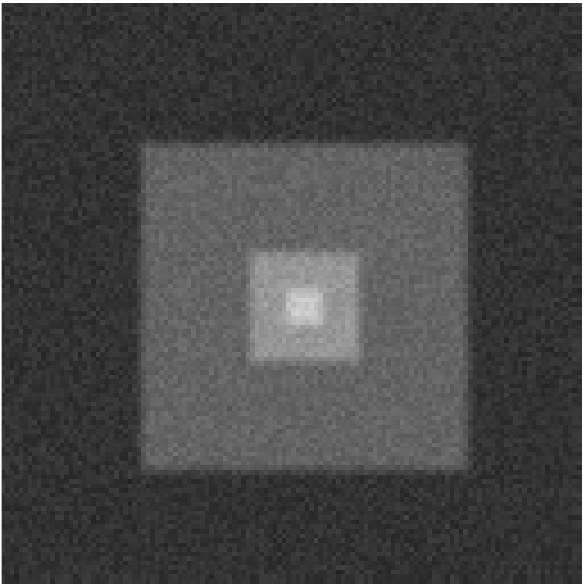}
            \subcaption{Blurred and noisy image. $\PSNR = 31.6$, $\SSIM = 0.14$}
    	\label{pic-squares_noisy}
    \end{subfigure}
\begin{subfigure}[t]{0.32\textwidth}
            \centering
            \includegraphics[width=\textwidth]{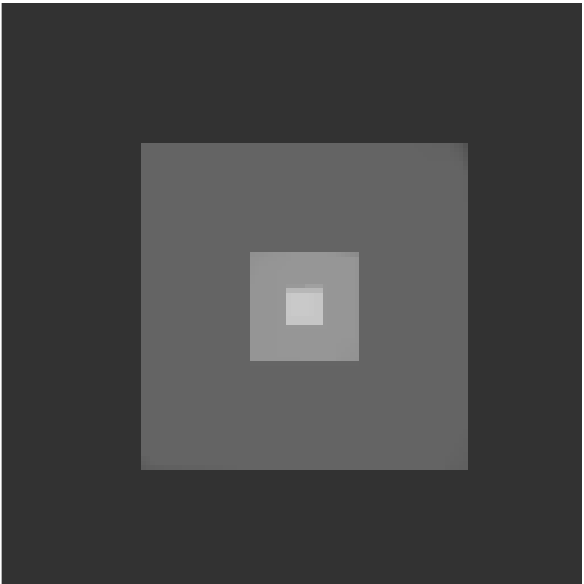}
            \subcaption{Reconstruction using the exact operator and isotropic $\TV$. $\PSNR = 52.9$, $\SSIM = 0.95$}
    	\label{pic-squares_exact}
    \end{subfigure}
    \\
\begin{subfigure}[t]{0.32\textwidth}
            \centering
            \includegraphics[width=\textwidth]{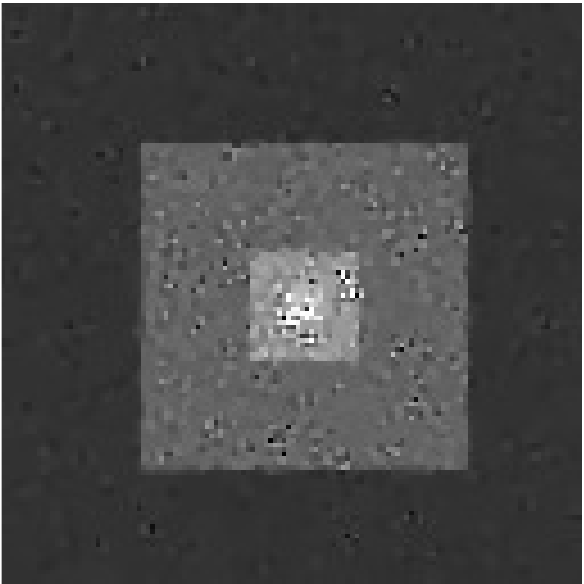}
            \subcaption{Reconstruction using a noisy operator and isotropic $\TV$. $\PSNR = 29.3$, $\SSIM = 0.16$}
    	\label{pic-squares_wrong}
    \end{subfigure}
\begin{subfigure}[t]{0.32\textwidth}
            \centering
            \includegraphics[width=\textwidth]{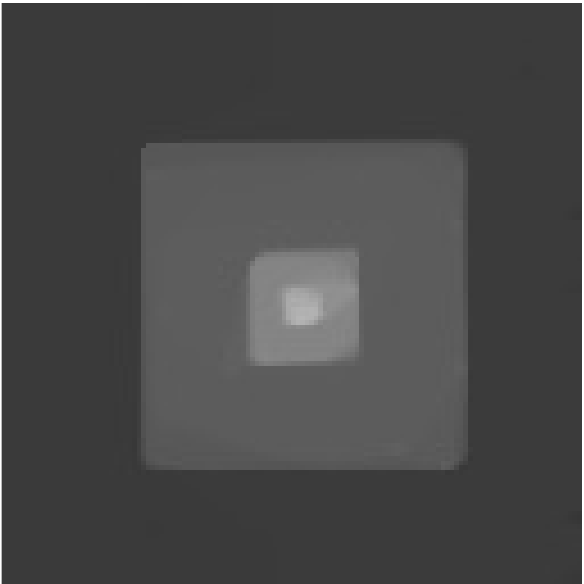}
            \subcaption{Reconstruction using interval bounds for the operator and isotropic $\TV$. $\PSNR = 28.6$, $\SSIM = 0.66$}
    	\label{pic-squares_bounds_isotropic}
    \end{subfigure}
\begin{subfigure}[t]{0.32\textwidth}
            \centering
            \includegraphics[width=\textwidth]{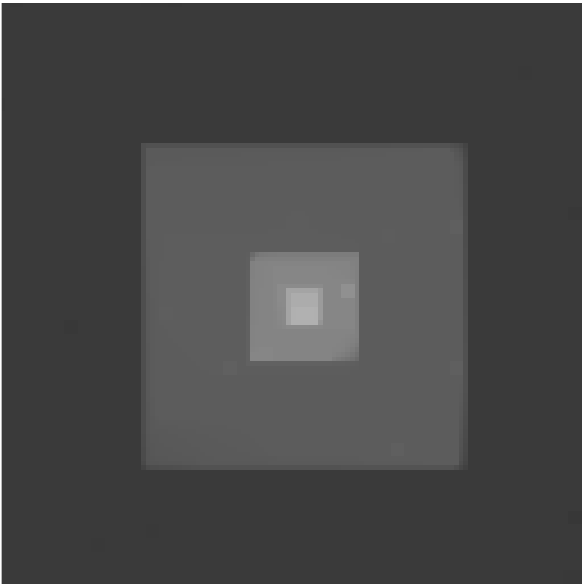}
            \subcaption{Reconstruction using interval bounds for the operator and anisotropic $\TV$. $\PSNR = 28.8$, $\SSIM = 0.72$}
    	\label{pic-squares_bounds_anisotropic}
    \end{subfigure}
    \caption{Reconstruction of a piecewise-constant image using total variation. Using a noisy operator produces clearly visible artifacts (\subref{pic-squares_wrong}). Reconstruction using interval bounds for the operator removes the artifacts at the price of a slight loss of contrast (\subref{pic-squares_bounds_isotropic}, \subref{pic-squares_bounds_anisotropic}). Isotropic $\TV$ favoures squares with rounded corners (\subref{pic-squares_bounds_isotropic}). Anisotropic $\TV$ leaves the corners sharp (\subref{pic-squares_bounds_anisotropic}).}
\end{figure}

Reconstruction using the exact operator yields nearly perfect results (Figure~\ref{pic-squares_exact}).

Let us assume, as in the previous section, that only a noisy version $\tilde A$ of the forward operator is available, which we obtain by adding $5\%$ uniform noise to the exact forward operator $A$:
\begin{equation}\label{bounds_matrix}
a^u_{ij} = \tilde a_{ij} + d, \quad a^l_{ij} = \max\{\tilde a_{ij} -d,0\} 
\end{equation}

\noindent with $d=0.025*\max_{k,l}\tilde a_{kl}$. Then we set all entries in $\tilde A$ less than $d$ to zero and normalise the rows of $\tilde A$.

Let us reconstruct the image using the noisy operator and isotropic total variation. The result is shown in Figure~\ref{pic-squares_wrong}. We observe numerous artifacts, however, contrast is mainly preserved and corners of the squares are sharp.

Let us now reconstruct the image using interval bounds for the forward operator, which we obtain the same way as in the one-dimensional example. The result is shown in Figure~\ref{pic-squares_bounds_isotropic}. The artifacts are removed, but the contrast is slightly reduced, which can be explained the same way as in the one-dimensional example. We also observe that the corners of the squares became rounded, which is the behaviour that we indeed would expect from isotropic $\TV$, which discourages sharp corners.

This illustrates an important feature of the interval-based approach. The feasible set~\eqref{U} obtained using interval bounds for the operator is larger and  gives more freedom to the regulariser than a feasible set obtained using a fixed operator (whether exact or noisy). Therefore, features specific to the regulariser are more apparent in the reconstructions in the former case, which explains the rounding of the corners in Figure~\ref{pic-squares_bounds_isotropic}. To support this idea, let us reconstruct the same image using anisotropic total variation, which does not penalise sharp corners~\cite{Esedoglu_Osher_isotropic_TV:2004}. The result is shown in Figure~\ref{pic-squares_bounds_anisotropic}. Sharp corners are recovered and even the contrast looks a bit better. 

Let us now consider another example shown in Figure~\ref{pic-horiz_orig}. This picture is more challenging because it contains much smaller objects, some of which have little contrast from the background, but it is still pieceiwse-constant, making total variation a suitable regulariser. 

\begin{figure}[t!]
    \centering
    \begin{subfigure}[t]{0.32\textwidth}
            \centering
            \includegraphics[width=\textwidth]{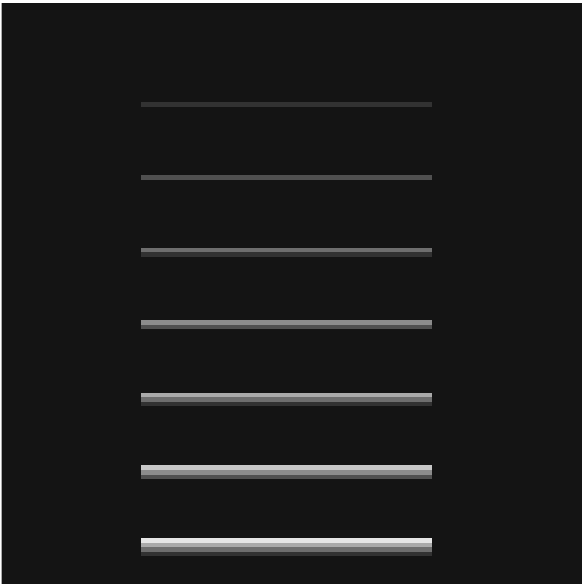}
            \subcaption{Original image}
    	\label{pic-horiz_orig}
    \end{subfigure}
\begin{subfigure}[t]{0.32\textwidth}
            \centering
            \includegraphics[width=\textwidth]{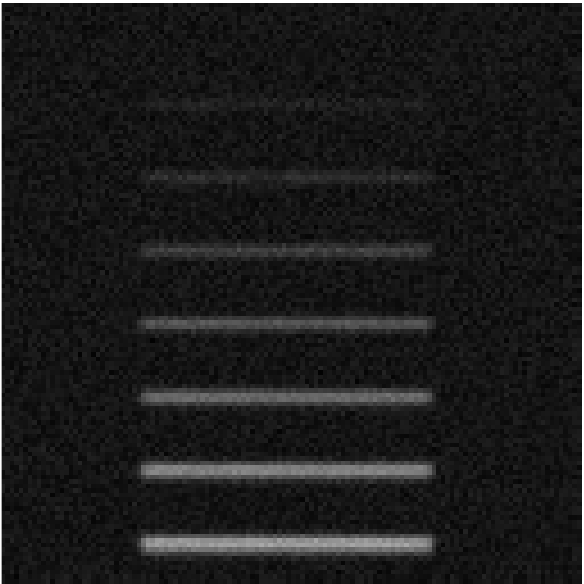}
            \subcaption{Blurred and noisy image. $\PSNR = 25.6$, $\SSIM = 0.20$}
    	\label{pic-horiz_noisy}
    \end{subfigure}
\begin{subfigure}[t]{0.32\textwidth}
            \centering
            \includegraphics[width=\textwidth]{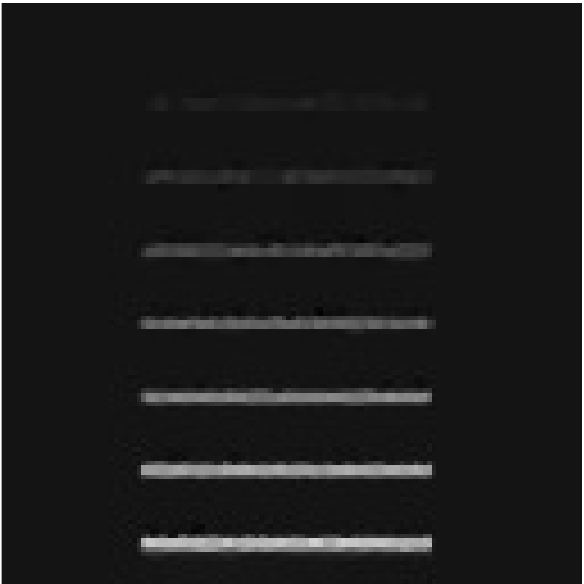}
            \subcaption{Reconstruction using the exact operator and isotropic $\TV$. $\PSNR = 29.1$, $\SSIM = 0.79$}
    	\label{pic-horiz_exact}
    \end{subfigure}
    \\
\begin{subfigure}[t]{0.32\textwidth}
            \centering
            \includegraphics[width=\textwidth]{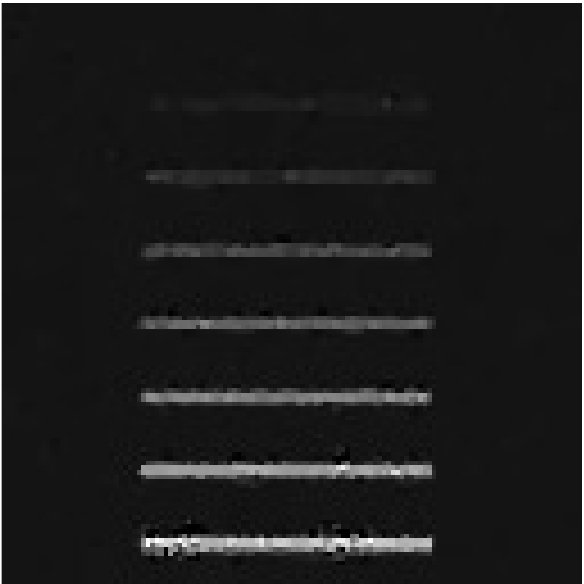}
            \subcaption{Reconstruction using a noisy operator and isotropic $\TV$. $\PSNR = 28.4$, $\SSIM = 0.33$}
    	\label{pic-horiz_wrong}
    \end{subfigure}
\begin{subfigure}[t]{0.32\textwidth}
            \centering
            \includegraphics[width=\textwidth]{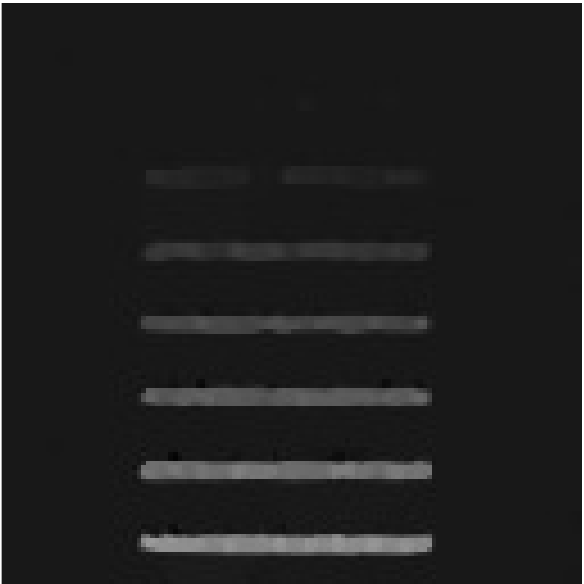}
            \subcaption{Reconstruction using interval bounds for the operator and isotropic $\TV$. $\PSNR = 24.2$, $\SSIM = 0.62$}
    	\label{pic-horiz_bounds_isotropic}
    \end{subfigure}
\begin{subfigure}[t]{0.32\textwidth}
            \centering
            \includegraphics[width=\textwidth]{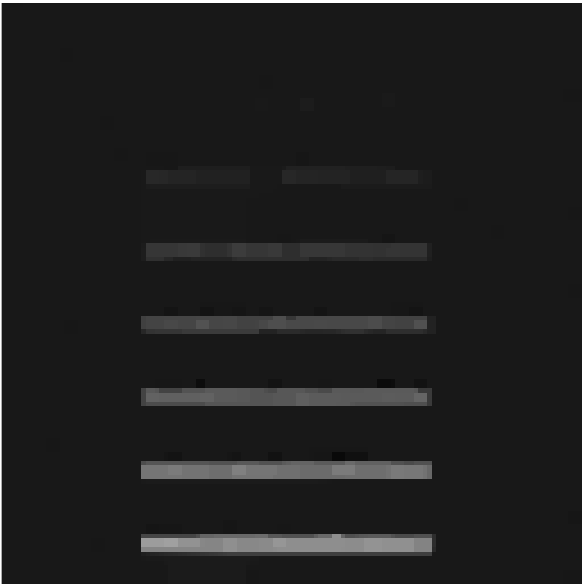}
            \subcaption{Reconstruction using interval bounds for the operator and anisotropic $\TV$. $\PSNR = 24.5$, $\SSIM = 0.65$}
    	\label{pic-horiz_bounds_anisotropic}
    \end{subfigure}
    \caption{Reconstruction of a piecewise-constant image  with thin structures using total variation. Even reconstruction using the exact operator produces some artifacts (\subref{pic-horiz_exact}). These artifacts are much more clearly present in the reconstruction using a noisy operator (\subref{pic-horiz_wrong}). Reconstruction using interval bounds for the operator removes the artifacts at the price of some loss of contrast (\subref{pic-horiz_bounds_isotropic}, \subref{pic-horiz_bounds_anisotropic}). The top thin line disappears. Anisotropic $\TV$ (\subref{pic-squares_bounds_anisotropic}) yields slightly better reconstruction than isotropic $\TV$ (\subref{pic-squares_bounds_isotropic}).}
\end{figure}

Already in the reconstruction using the exact operator (Figure~\ref{pic-horiz_exact}) there are some artifacts and the top thin line almost disappeared. The reconstruction quality is even worse when we use the noisy operator (Figure~\ref{pic-horiz_wrong}). Artifacts are clearly visible, especially in the bottom line. 

Reconstruction using interval bounds for the operator removes the artifacts (Figure~\ref{pic-horiz_bounds_isotropic}), but causes some lost of contrast, as we already observed in previous examples. The top line, which already had little contrast from the background in the original image, completely disappears in the reconstruction. Using anisotropic total variation further removes the artifacts, but does not recover the top line (Figure~\ref{pic-horiz_bounds_anisotropic}). 

This loss of contrast and especially the disappearance of some details are, of course, unwanted features of a reconstruction algorithm. However, it is important to note that the solutions shown in Figures~\ref{pic-horiz_bounds_isotropic} and~\ref{pic-horiz_bounds_anisotropic} indeed could produce the observed data in Figure~\ref{pic-horiz_noisy} (as opposed to the reconstruction in Figure~\ref{pic-horiz_wrong}). Therefore, we have no reasons to reject the images in Figures~\ref{pic-horiz_bounds_isotropic} and ~\ref{pic-horiz_bounds_anisotropic} as plausible reconstructions. By Theorem~\ref{thm-U_in_U*} we have that there is a blurring operator $A$ between the interval bounds $A^l$ and $A^u$ and a blurry image $f$ between $f^l$ and $f^u$ such that $Au=f$, where $u$ is any element of the feasible set~\eqref{U}, e.g., the images shown in Figure~\ref{pic-horiz_bounds_isotropic} or~\ref{pic-horiz_bounds_anisotropic}. The only way to reject these images would be by narrowing the feasible set by adding more information about the unknown solution or the unknown forward operator. We attempted the latter in Section~\ref{section-U**}, but came to the conclusion that adding a linear constraint on the forward operator breaks the convexity of the feasible set.

\paragraph{Deblurring of natural images} Let us finally assess the performance of the proposed method on real images. Consider a well-known example, the cameraman (Figure~\ref{pic-cameraman_orig}). Let us add blur and noise to it the same way as in the previous examples (Figure~\ref{pic-cameraman_noisy}).

\begin{figure}[t!]
    \centering
    \begin{subfigure}[t]{0.32\textwidth}
            \centering
            \includegraphics[width=\textwidth]{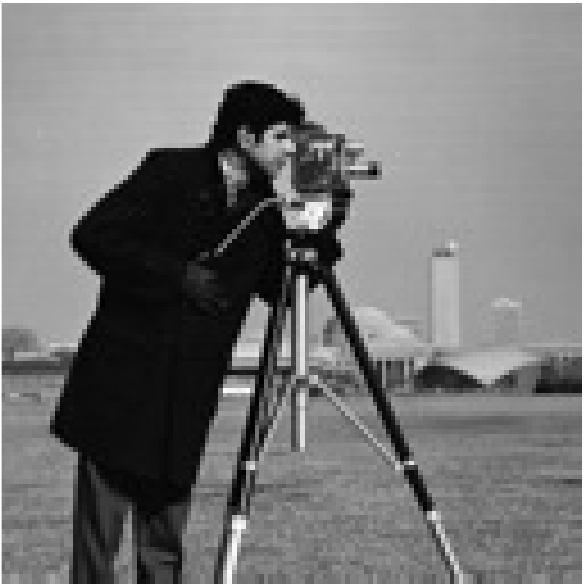}
            \subcaption{Original image}
    	\label{pic-cameraman_orig}
    \end{subfigure}
\begin{subfigure}[t]{0.32\textwidth}
            \centering
            \includegraphics[width=\textwidth]{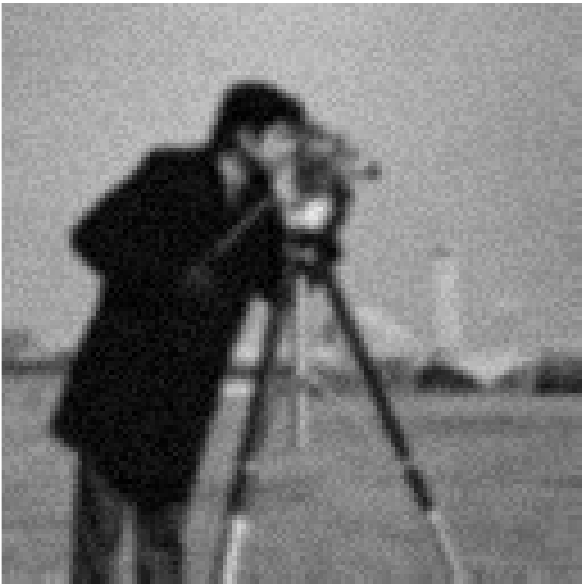}
            \subcaption{Blurred and noisy image. $\PSNR = 25.0$, $\SSIM = 0.41$}
    	\label{pic-cameraman_noisy}
    \end{subfigure}
\begin{subfigure}[t]{0.32\textwidth}
            \centering
            \includegraphics[width=\textwidth]{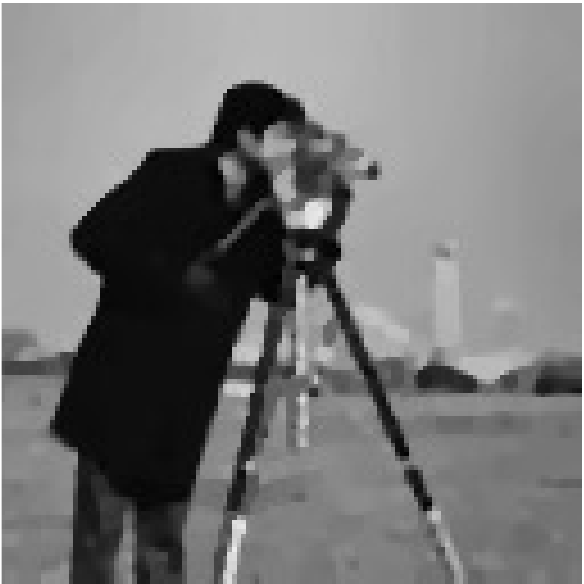}
            \subcaption{Reconstruction using the exact operator and isotropic $\TV$. $\PSNR = 27.5$, $\SSIM = 0.60$}
    	\label{pic-cameraman_exact}
    \end{subfigure}
    \\
\begin{subfigure}[t]{0.32\textwidth}
            \centering
            \includegraphics[width=\textwidth]{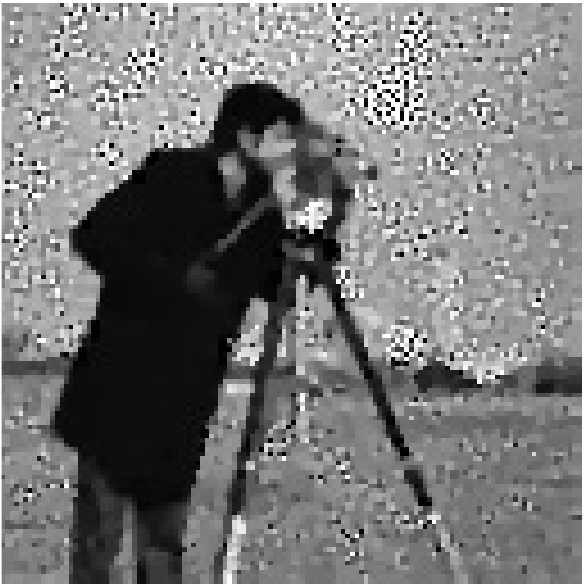}
            \subcaption{Reconstruction using a noisy operator and isotropic $\TV$. $\PSNR = 18.7$, $\SSIM = 0.36$}
    	\label{pic-cameraman_wrong}
    \end{subfigure}
\begin{subfigure}[t]{0.32\textwidth}
            \centering
            \includegraphics[width=\textwidth]{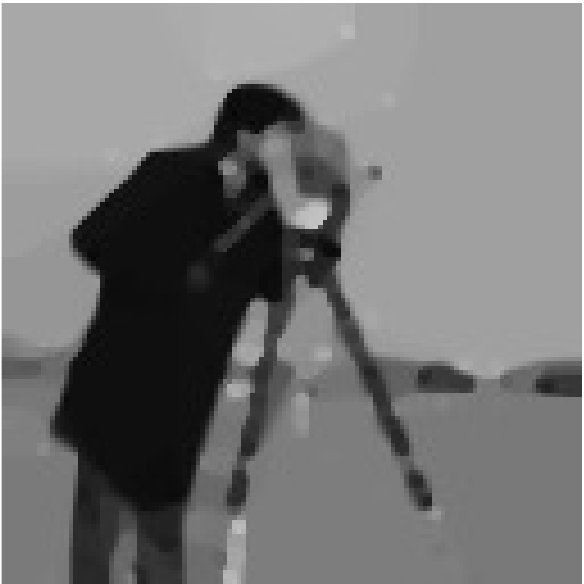}
            \subcaption{Reconstruction using interval bounds for the operator and isotropic $\TV$. $\PSNR = 23.0$, $\SSIM = 0.27$}
    	\label{pic-cameraman_bounds_isotropic}
    \end{subfigure}
\begin{subfigure}[t]{0.32\textwidth}
            \centering
            \includegraphics[width=\textwidth]{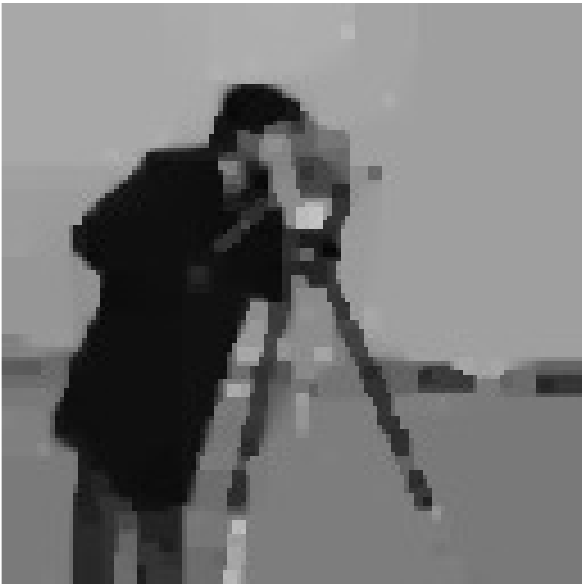}
            \subcaption{Reconstruction using interval bounds for the operator and anisotropic $\TV$. $\PSNR = 22.8$, $\SSIM = 0.26$}
    	\label{pic-cameraman_bounds_anisotropic}
    \end{subfigure}
    \caption{Reconstruction of a natural image using total variation. Reconstruction using a noisy operator produces numerous artifacts (\subref{pic-cameraman_wrong}). In the reconstructions using interval bounds for the operator the artifacts disappear (\subref{pic-cameraman_bounds_isotropic}, \subref{pic-cameraman_bounds_anisotropic}), but fine details are also removed and the images look cartoon-like.}
\end{figure}

Reconstruction using the exact operator (Figure~\ref{pic-cameraman_exact}) is of reasonable quality, with some loss of fine details, especially in the face of the cameraman and in the camera. Reconstruction using the noisy operator contains numerous artifacts (Figure~\ref{pic-cameraman_wrong}). In the reconstructions using interval bounds for the operator the artifacts disappear, but also most of the fine details of the image are removed and the image looks very cartoon-like, both with isotropic and anisotropic $\TV$ (Figures~\ref{pic-cameraman_bounds_isotropic} and~\ref{pic-cameraman_bounds_anisotropic}). 

The reason for the observed behaviour is additional freedom that the larger feasible set gives to the regulariser. Since natural images are rarely piecewise-constant, the piecewise-constant reconstructions that $\TV$ promotes don't look natural. The idea of the approach using interval bounds for the operator is to use less information coming from the forward operator, which is not perfectly known, and instead rely more on the regulariser, which is supposed to reasonably encode prior information about the image. If this is not the case, the approach results in reconstructions that look quite different from the original. Therefore, in order to use the proposed approach on natural images, one needs to try more appropriate regularisers, e.g., $\TGV$~\cite{bredies2009tgv} or $TVL_p$~\cite{Burger_TVLp_2016}. This step is, however, beyond the scope of the present paper.

\section{Conclusions}

In this paper we analysed an approach to image reconstruction problems with uncertainty in the forward operator based on partially ordered spaces. The method is essentially a variant of the residual method with a feasible set based on order intervals. Our main theoretical contribution is the study of this feasible set. It turned out that the feasible set admits two equivalent descriptions, one of which could be modified to include additional \emph{a priori} constraints on the forward operator. This is especially relevant in deblurring applications, since the rows of a blurring matrix must always sum up to one, which can be expressed as a linear constraint on the matrix. Unfortunately, we came to the conclusion that adding a linear constraint on the forward operator breaks the convexity of the feasible set.

Despite this negative result, we demonstrated in our numerical experiments that the approach is useful in deblurring with an imperfectly known blurring operator, especially in the case when the regulariser well captures qualitative information about the image. Our experiments revealed an important feature of the approach. In the absence of perfect knowledge of the forward operator, the only source of information that can compensate for this lack of knowledge is the regulariser. Therefore, features specific for the regulariser are more apparent in the interval-based method then in standard methods.


\printbibliography

\end{document}